\documentclass[a4paper,12pt,reqno]{amsart}
\usepackage{amsfonts}
\usepackage{mathrsfs}
\usepackage{amsmath,amssymb,latexsym,amsfonts,amscd, yfonts}
\usepackage[all,ps,cmtip]{xy}

\title{Explicit birational geometry of 3-folds and 4-folds of general type, III}
\author{Jungkai A. Chen and Meng Chen}

\address{\rm National Center for Theoretical Sciences, Taipei Office, and Department of Mathematics, National Taiwan University, Taipei, 106, Taiwan}
\email{jkchen@math.ntu.edu.tw}

\address{\rm Institute of Mathematics \& LMNS, Fudan University,
Shanghai, 200433, People's Republic of China}
\email{mchen@fudan.edu.cn}

\thanks{The first author was partially supported by NCTS/TPE and
National Science Council of Taiwan. The second author was
supported by National Natural Science Foundation of China (\#11171068, \#11121101, \#11231003) and Doctoral Fund of Ministry of Education of China (\#20110071110003)}

\newcommand{\bQ}{{\mathbb Q}}
\newcommand{\bP}{{\mathbb P}}
\newcommand{\roundup}[1]{\lceil{#1}\rceil}
\newcommand{\rounddown}[1]{\lfloor{#1}\rfloor}

\newcommand\Vol{\text{\rm Vol}}

\newcommand\OO{{\mathcal{O}}}

\newcommand\lsleq{{\preceq}}
\newcommand\lsgeq{{\succeq}}
\newcommand\lrw{\longrightarrow}

\newcommand\tl{\theta_\Lambda}
\newcommand\ttm{\theta_{m_0}}

\newtheorem{thm}{Theorem}[section]
\newtheorem{lem}[thm]{Lemma}
\newtheorem{cor}[thm]{Corollary}
\newtheorem{prop}[thm]{Proposition}

\newtheorem{op}[thm]{Open problem}
\theoremstyle{definition}
\newtheorem{defn}[thm]{Definition}

\newtheorem{exmp}[thm]{Example}

\newtheorem{rem}[thm]{Remark}
\theoremstyle{remark}

\newtheorem{prbm}[thm]{\bf Problem}
\begin{document}
\numberwithin{equation}{section}
\begin{abstract}
Nonsingular projective 3-folds $V$ of general type can be
naturally classified into 18 families according to the {\it
pluricanonical section index} $\delta(V):=\text{min}\{m|P_m\geq
2\}$ since $1\leq \delta(V)\leq 18$ due to our previous series (I,
II). Based on our further classification to 3-folds with
$\delta(V)\geq 13$ and an intensive geometrical investigation to
those with $\delta(V)\leq 12$, we prove that $\Vol(V) \geq
\frac{1}{1680}$ and that the pluricanonical map $\Phi_{m}$ is
birational for all $m \geq 61$, which greatly improves known
results. An optimal birationality of $\Phi_m$ for the case
$\delta(V)=2$ is obtained. As an effective application, we study
projective 4-folds of general type with $p_g\geq 2$ in the last
section.
\end{abstract}
\maketitle

\pagestyle{myheadings} \markboth{\hfill J. A. Chen and M. Chen
\hfill}{\hfill Explicit birational geometry of 3-folds and 4-folds\hfill}

\section{\bf Introduction}
One of the fundamental aspects of  birational geometry is to
understand the behavior of the natural pluricanonical map $\Phi_m$ of any
variety for any $m\in {\mathbb Z}_{>0}$. The induced fibrations 
possibly reduce the studies to lower dimensional situations. Varieties of general type, which are those with birational pluricanonical maps $\Phi_m$  for  sufficiently large $m$, are therefore considered as the basic building blocks of varieties.

  For varieties of general
type, a key problem is  to find an effective integer $m>0$ so that
$\Phi_m$ is birational.  The remarkable theorem of Hacon and McKernan
\cite{H-M}, Takayama \cite{Tak}, and Tsuji \cite{Tsuji} says that
there is a constant $c(n)$ so that $\Phi_m$ is birational for all
$n$-dimensional varieties of general type and for all $m\geq
c(n)$.  However, these constants are explicitly known only when $n \le 3$.

 In fact, the problem is almost equivalent to find a practical lower bound of the
canonical volume which computes the rate of growth of plurigenera, or equivalent to find $m_0$ such that plurigenus $P_{m_0}$ is sufficiently large.
 One may also refer to the  nice survey article
Hacon--McKernan \cite{HM10} for various boundedness results in
birational geometry.

The motivation of this series is to study birational geometry of 3-folds and higher dimensional varieties of general type. The main purpose is to investigate the following:

\begin{op} Find optimal constants $v_3\in {\mathbb Q}_{>0}$ and $b_3\in {\mathbb Z}_{>0}$ so that,  for all nonsingular projective 3-folds $V$ of general type,
\begin{itemize}
\item[i.] $\Vol(V)\geq v_3$ and
\item[ii.] $\Phi_{m}$ is birational for all $m\geq b_3$.
\end{itemize}
\end{op}

 Recall that we have proved the following:

\begin{thm}(\cite[Theorems 1.1, 1.2]{Explicit2}) \label{Exp2} Let $V$ be a nonsingular projective 3-fold of general type. Then
\begin{itemize}
\item[1.] $\Vol(V)\geq \frac{1}{2660}$;
\item[2.] there exists a positive integer $m_0(V)\leq 18$ so that $P_{m_0}\geq 2$;
\item[3.] the pluricanonical map $\Phi_{m}$ is birational onto its image for all $m\geq 73$.
\end{itemize}
\end{thm}
For more results on explicit birational geometry of 3-folds of general type, one may refer to our previous papers \cite{Explicit1, Explicit2}.

In order to  formulate our main statements of this article, we need to recall some general results and introduce some definition. Given a projective variety $V$ of general type, there exists a minimal model  $X$ birational to $V$(cf. \cite{BCHM}).
Thanks to the Riemann-Roch formula and Vanishing Theorem, $\Vol(V)=K_X^{\dim X}$. Notice that in dimension three or higher, a minimal model may have singularities. Hence $K_X^{\dim X}$ is just a positive rational number.

 A minimal model has at worst terminal singularities. In dimension three, terminal singularities was  classified by Mori. A three dimensional terminal singularity is one of the following: a  terminal quotient singularity of type $\frac{1}{r}(1,-1,b)$ for some $b$ relatively prime to $r$ which we usually denote it as $(b,r)$ for short, an isolated cDV point, a quotient of an isolated cDV point. It is well-known to experts that a three dimensional terminal point can be deformed into a collection of terminal quotient singularities, which is called  {\it basket of singularities}. An important feature of three dimensional birational geometry is the Singular Riemann-Roch formula due to Reid \cite{YPG}
 $$ \chi(X, mK_X) = \frac{m(m-1)(2m-1)K_X^3}{12}+(1-2m)\chi(X, \mathcal{O}_X) +l_m,$$
 where $l_m$ denotes the contribution of singularities which can be computed by baskets.
 It follows that all plurigenera and hence canonical volume of a minimal 3-fold $X$ are completely determined by $P_2(X)$, $\chi(X, \mathcal{O}_X)$ and baskets of singularities $B_X$, of which we called such a triple {\it the weighted basket} of $X$.
 For the basic properties of  weighted baskets, one may refer to \cite[Section 3]{Explicit1}. Since our problems are birational in nature, the studies of nonsingular threefold $V$ is equivalent to the studies of its minimal model $X$. In particular, we may and do consider  the weighted basket of $V$ as the weighted basket of its minimal model $X$. \footnote{Even though minimal models are not necessarily unique, it is known that two birational minimal models are connected by flops (cf. \cite{Ka08}). Together with the fact that a 3-dimensional flop preserves singularity types (cf. \cite{Ko89}), it follows that baskets of $V$ is independent of choices of minimal models.}

Next, we would like to define the {\it
pluricanonical section index} (or, in short, the {\it ps-index})
$$\delta(V):=\text{min}\{m|m\in {\mathbb Z}_{>0},\ P_m(V)\geq 2\},$$
which is clearly a birational invariant. By Theorem \ref{Exp2}, we have
$\delta(V)\leq 18$ for any 3-fold $V$ of general type.
Note that 3-folds $V$ with $\delta(V)=1$ (i.e., $p_g(V)\geq 2$) have been intensively studied in \cite{IJM, MA}
 where optimal results are realized.
 Threefolds of general type with $\delta(V)\geq 2$ are far from being clear.  Sometimes we use the symbol $\delta(X)$ directly since $X$ is birationally equivalent to $V$.

 \begin{exmp}
The ``worst''  known minimal 3-fold is the weighted hyper-surface $X:=X_{46}\subset \bP(4,5,6,7,23)$
(cf. \cite{Fletcher}) which has the invariants: $\delta(X)=10$ and
$\Vol(X)=K_X^3=\frac{1}{420}$. Also $\Phi_{26}$ is not birational.
\end{exmp}

In this paper, we mainly investigate projective $3$-folds of general type with $\delta(V)\geq 2$. Our main results are as follows.

\begin{thm}\label{M1} (=Theorem \ref{-13}) Let $V$ be a nonsingular projective 3-fold of general type with $\delta(V) \geq 13$. Then its weighted basket $\mathbb{B}=\{B_V, P_2(V), \chi(\OO_V)\}$ belongs to one of the types in {\rm Tables F--0, F--1, F--2} in Appendix and the following is true:
\begin{itemize}
\item[(1)] $\delta(V)=18$ if and only if $\mathbb{B}(V)=\{B_{2a},  0,  2\}$;
\item[(2)] $\delta(V)\neq 16, \ 17$;
\item[(3)] $\delta(V)=15$ if and only if $\mathbb{B}(V)$ belongs to one of the types in {\rm Table F--1};
\item[(4)] $\delta(V)=14$ if and only if $\mathbb{B}(V)$ belongs to one of the types in {\rm Table F--2};
\item[(5)] $\delta(V)=13$ if and only if $\mathbb{B}(V)=\{B_{41}, 0, 2\}$,
\end{itemize}
where $B_{2a}$ and $B_{41}$ can be found in {\rm Table F--0}
\end{thm}

Some other results for 3-folds with large $\delta(V)$ are given in Section 4. For example, one has

\begin{cor}\label{<<420} (=Corollary \ref{<420}) Let $V$ be a nonsingular projective 3-fold of general type with
$\Vol(V)<\frac{1}{336}$. Then $\delta(V)\geq 8$.
\end{cor}

We also prove the following:

\begin{thm}\label{M2} Let $V$ be a nonsingular projective 3-fold of general type. Then
\begin{itemize}
\item[(1)] $\Phi_m$ is birational for all $m\geq 61$;
\item[(2)]  $\Vol(V)\geq \frac{1}{1680}$. Furthermore, $\Vol(V)=\frac{1}{1680}$ if and only if ${\mathbb B}(V)=\{B_{7a}, 0, 2\}$ or $\{B_{36a}, 0, 2\}$, where
$B_{7a}$ and $B_{36a}$ can be found in {\rm Table F--2}
\end{itemize}

\end{thm}

A direct by-product of our method is the following:

\begin{cor}\label{pg1} Let $V$ be a nonsingular projective 3-fold of general type with $p_g(V)=1$. Then
\begin{itemize}
\item[(1)] $\Vol(V)\geq \frac{1}{75}$;
\item[(2)] $\Phi_m$ is birational for all $m\geq 18$.
\end{itemize}
\end{cor}

In the second part of this paper we prove some optimal results on 3-folds with $\delta(V)=2$.

\begin{thm}\label{1/2} Let $V$ be a nonsingular projective 3-fold of general type with $\delta(V)\leq 2$. Then
\begin{itemize}
\item[(1)] $\Phi_m$ is birational for all $m\geq 11$;
\item[(2)] If $\Phi_{10}$ is not birational, then
 $0\leq \chi(\OO_V)\leq 3$ and $|2K_V|$
is composed of a rational pencil of $(1,2)$ surfaces. Furthermore,  $\#\{{\mathbb B}(V)\}<+\infty$ and the initial basket $B^0$ of $B_V$ belongs to one of the types in {\rm Tables II--1, II--2, II--3} in the Appendix.
\end{itemize}
\end{thm}

The following examples show that our results in Theorem \ref{1/2} are optimal.

\begin{exmp}\label{ex} (Iano-Fletcher \cite[P. 151, P. 153]{Fletcher})
\begin{itemize}
\item[(1)] General weighted complete intersections
$X_{22}\subset \mathbb{P}(1,2,3,4,11)$ and $X_{6,18}\subset \mathbb{P}(2,2,3,3,4,9)$
both have  ps-index $\delta=2$. Since both  $X_{22}$ and
$X_{6,18}$ have non-birational $10$-canonical map, Theorem
\ref{1/2}(1) is optimal.

\item[(2)] The 3-fold $X_{22}$ corresponds to  No. 1 in {\rm Table II--1} with $\chi=0$ and $X_{6,18}$ belongs to  No. 11 (with $t=1$) in {\rm Table II--1}.
\end{itemize}\end{exmp}

\begin{rem} Theorem \ref{1/2} is parallel to main results in \cite{IJM}. We have similar statements to Theorem \ref{1/2} for  3-folds with $\delta(V)\geq 3$. We omit them since we are not sure whether they are optimal or not.
\end{rem}

In the last part we study projective 4-folds. The main result is
the following:

\begin{thm}(=Theorem \ref{b4})\label{4folds} Let $V$ be a nonsingular projective 4-fold of general type. Then,
\begin{itemize}
\item[(i)] when $p_g(V)\geq 2$, $\Phi_{|mK_V|}$ is birational for all $m\geq 35$;
\item[(ii)] when $p_g(V)\geq 19$, $\Phi_{|mK_V|}$ is birational for all $m\geq 18$.
\end{itemize}
 \end{thm}

\medskip

This paper is organized as follows.  In Section 2, we start with
general setting on rational maps on  varieties of general type
and review some known useful inequalities. Then we list several
basic lemmas on 3-folds. In Section 3, we  improve our technique
used in \cite{Explicit2} to bound $K_X^3$ from below. Applying our
basket analysis developed in \cite{Explicit1}, we obtain an
effective function $v(x)$ in Section 4 so that $K_X^3\geq
v(\delta(X))$ for any given minimal 3-fold $X$. Section 5 is
devoted to compiling the clean list for ${\mathbb B}(X)$ with
$\delta(X)\geq 13$. Then, in Section 6, we are able to study the
birationality of $\Phi_m$. Section 7 is dedicated to classifying
3-folds with $\delta=2$. {}Finally we study nonsingular
projective 4-folds of general type with $p_g\geq 2$ in Section 8. All
subsidiary tables are presented in the Appendix.

\bigskip

Throughout we work over any algebraically closed field $k$ of characteristic 0. 
We are in favor of the following symbols:
\begin{itemize}
\item[$\circ$] ``$\sim$'' denotes linear equivalence or
${\bQ}$-linear equivalence; \item[$\circ$] ``$\equiv$'' denotes
numerical equivalence; \item[$\circ$] ``$|A|\ \lsleq\ |B|$'' means
that $|B|\supseteq |A|+\text{fixed effective divisors}$.
\end{itemize}

\section{\bf Preliminaries}

We begin with the  general setting on  rational maps defined
by some sub-linear system of the pluricanonical system $|mK|$ on
varieties of general type.  Let $V$ be any nonsingular projective
variety of general type with dimension $n\geq 3$.  According to
the Minimal Model Program,   $V$ has a minimal model (see e.g. \cite{KMM}, \cite{K-M},
\cite{BCHM} and \cite{Siu}). {}From the point of view of
birational geometry, we may always consider the rational map on
minimal varieties of general type. A minimal model
$X$ is a normal projective variety with a nef
canonical divisor $K_{X}$ and with $\bQ$-factorial terminal
singularities.

\subsection{The rational map $\Phi_{\Lambda}$ for $\Lambda\subset |m_0K|$}\label{setup}

Let $X$ be a minimal projective variety of general type on which $P_{m_0}(X)\geq 2$ for a positive integer $m_0$.  Let $\Lambda\subset |m_0K_X|$ be a positive dimensional linear system.
{}Fix an effective Weil
divisor $K_{m_0}\sim m_0K_X$ on $X$. Take successive blow-ups
$\pi\colon X'\rightarrow X$ along nonsingular centers, such that the
following conditions are satisfied:
\medskip
\begin{itemize}
\item[(i)] $X'$ is smooth;


\item[(ii)] the moving part  of $\pi^*(\Lambda)$ is base point
free and so that $g:=\Phi_{\Lambda}\circ\pi$ is a non-constant
morphism;

\item[(iii)] $\pi^*(K_{m_0})
\cup\{\pi-\text{exceptional divisors} \}$ has simple normal crossing
supports.

\end{itemize}

Sometimes we will take further blow-ups so that $\pi$ satisfies
some more conditions, which will be specified explicitly.

We have a morphism $g\colon X'\longrightarrow \overline{\Phi_{\Lambda}(X)}\subseteq{\mathbb
P}^{N}$. Let $X'\overset{f}\longrightarrow
\Gamma\overset{s}\longrightarrow \overline{\Phi_{\Lambda}(X)}$ be the Stein factorization of $g$.
We have the following commutative diagram:

\begin{eqnarray*}
\xymatrix{ X' \ar[rr]^{f}\ar[d]_{\pi }\ar[drr]^{g}&& \Gamma \ar[d]^{s}\\
X \ar@{.>}[rr]_{\Phi_\Lambda} && \overline{\Phi_{\Lambda}(X)}}
\end{eqnarray*}



We may write $m_0K_{X'}=_{\mathbb Q}\pi^*(m_0K_X)+E_{\pi, m_0}$
where $E_{\pi, m_0}$ is an effective $\pi$-exceptional ${\mathbb
Q}$-divisor. Denote by $M_{m_0}$ (resp. $M_\Lambda$) the movable
part of $|m_0K_{X'}|$ (resp. $\pi^* \Lambda$).  Set $d_{m_0}:=
\dim \Phi_{m_0}(X)$ (resp. $d_{\Lambda}:=\dim \Gamma$). The
Bertini Theorem implies that the general member of the moving part
$M_{\Lambda}$ of $\pi^*(\Lambda)$ is irreducible whenever
$d_{\Lambda}\geq 2$ and, otherwise, $M_{\Lambda}\equiv
a_{\Lambda}F$, where $a_{\Lambda}:=\deg f_*\OO_{X'}(M_{\Lambda})$
and $F$ is a general fiber of $f$. We set
$$\theta_\Lambda:=\begin{cases}
1,&\ \text{if}\ d_{\Lambda}\geq 2;\\
a_{\Lambda},&\ \text{if}\ d_{\Lambda}=1.
\end{cases}$$
Recall our definition in \cite[Definition 2.4]{Explicit2}, the {\it generic irreducible element} $\Sigma$ of $\pi^*(\Lambda)$ is defined as follows:
$$\Sigma_\Lambda:=\begin{cases}
\text{the general member of the moving part of}\ \pi^*(\Lambda) , &\ \text{if } d_{\Lambda}\geq 2;\\
F, &\ \text{if} \  d_{\Lambda}=1.
\end{cases}$$
By the above setting, we always have
$$m_0\pi^*(K_X)\sim_{\bQ} \theta_\Lambda \Sigma_\Lambda+E_{\Lambda}'$$
for some effective $\bQ$-divisor $E_{\Lambda}'$ on $X'$.
\medskip

\noindent {\bf Convention.} Whenever we are working on the complete
linear system $|m_0 K_X|$, we will use parallel notations such as $d_{m_0}$,
$\ttm$, $\cdots$ (or even just $d, \theta,\cdots$, for simplicity).
\medskip

We discuss the special case with $d_{\Lambda}=1$. Clearly the
general fiber $F$ is nonsingular projective of dimension
$\dim(X)-1$. Replace $X'$ by its birational model, we may assume
that there is a birational contraction morphism $\sigma:F\lrw F_0$
onto a minimal model $F_0$.  We have the following ``canonical
restriction inequality'':

\begin{lem}\label{cr} Keep the above settings. Suppose that $d_{\Lambda}=1$.  The following holds:
\begin{itemize}
\item[(i)] if $b:=g(\Gamma)>0$, then $\pi^*(K_X)|_F\sim \sigma^*(K_{F_0})$;
\item[(ii)] if $b=0$, then
$$\pi^*(K_X)|_F\geq \frac{ \tl }{m_0+ \tl}\sigma^*(K_{F_0}).$$
\end{itemize}
\end{lem}
\begin{proof} Statement (i) follows from Chen \cite[Lemma 2.5]{ASPM}.

Assume $\Gamma\cong \bP^1$.   Choose a sufficiently large and
divisible integer $m$ so that both $|m\pi^*(K_X)|$ and
$|mK_{F_0}|$ are base point free. By Kawamata's extension theorem
\cite[Theorem A]{EXT}, we have the surjective map:
$$H^0(X', m \tl (K_{X'}+F))\lrw H^0(F, m \tl K_F).$$
Since $|m(\tl +m_0)K_{X'}|\lsgeq |m \tl (K_{X'}+F)|$,
$\text{Mov}|m \tl K_F|=|m \tl \sigma^*(K_{F_0})|$ and
$|m(\tl+m_0)\pi^*(K_X)|=|M_{m( \tl+m_0)}|$, we obtain the
following inequality:
$$m(\tl+m_0)\pi^*(K_X)|_F= M_{m(\tl+m_0)}|_F\geq m \tl \sigma^*(K_{F_0}),$$
which implies (ii).
\end{proof}

\subsection{Key inequalities on 3-folds}\label{3fold}  Let $X$ be minimal 3-fold of general type. Assume that $\Lambda\subset
|m_0K_X|$ is a linear system of positive dimension.
As in \ref{setup}, we obtain an induced fibration $f:X'\lrw
\Gamma$. Pick a generic irreducible element $S$ of $|m_0K_{X'}|$.
Let $|G|$ be a given base point free linear system on $S$. Pick a
generic irreducible element $C$ of $|G|$. Since $\pi^*(K_X)|_S$ is
nef and big, Kodaira's lemma implies that $\pi^*(K_X)|_S\geq \beta
C$ for some rational number $\beta>0$. Then, by \cite[Inequality
(2.1)]{Explicit2}, one has
\begin{equation}\label{i1}
K_X^3\geq \frac{ \theta \beta}{m_0}\xi
\end{equation}
where $\xi:=(\pi^*(K_X)\cdot C)_{X'}$. Besides, by \cite[Remark 2.12]{Explicit2}, one has
\begin{equation}\label{i2}
\xi\geq \frac{\deg(K_C)}{1+\frac{m_0}{\theta}+\frac{1}{\beta}}.
\end{equation}
For any positive integer $m$ so that
$\alpha_m:=(m-1-\frac{m_0}{\theta}-\frac{1}{\beta})\xi>1$, by
Chen--Zuo \cite[Theorem 3.1]{Chen-Zuo}, one has
\begin{equation}\label{i3}
\xi\geq \frac{\deg(K_C)+\roundup{\alpha_m}}{m}.
\end{equation}

We have a stronger form of Inequality (\ref{i3}) when $C$ is
``even'':

\begin{lem}\label{eveni} Under the above situation, if $C$ is an even divisor on $S$ (i.e. $\frac{1}{2}C\in \text{Pic}(S)$),  then, for any $m>0$ so that $\alpha_m>0$,  one has
\begin{equation}\label{i3'}
\xi\geq \frac{\deg(K_C)+2\roundup{\frac{1}{2}\alpha_m}}{m}.
\end{equation}
\end{lem}
\begin{proof}  We refer to the proof for Chen--Zuo \cite[Theorem 3.1]{Chen-Zuo}.  The key point is to estimate $\deg(D)$ where $D=\roundup{Q}|_C$ and $Q$ is a $\bQ$-divisor on $S$ with $(Q\cdot C)=\alpha_m$.
Since $\deg(D)\geq \alpha_m>0$ and $\deg(D)$ is even, we naturally have
$$\deg(D)=2(\roundup{Q}\cdot \frac{1}{2}C)\geq 2\roundup{\frac{1}{2}\alpha_m}$$
where we note that  $(\roundup{Q}\cdot \frac{1}{2}C)$ is a positive integer. Clearly the rest of the proof of Chen--Zuo \cite[Theorem 3.1]{Chen-Zuo} implies Inequality (\ref{i3'}).
\end{proof}

When $d_{\Lambda}=1$, Lemma \ref{cr}(ii) implies the following:
\begin{equation} \xi=(\pi^*(K_X)\cdot C)_{X'}\geq \frac{\theta}{m_0+\theta}(\sigma^*(K_{F_0})\cdot C)_F.
\end{equation}

\subsection{Other useful Lemmas}
\begin{lem}(see Ma\c{s}ek \cite[Proposition 4]{Masek} or \cite[Lemma 2.6]{2011NG2})\label{masek} Let $S$ be a nonsingular projective surface. Let $L$ be a nef and big $\bQ$-divisor on $S$ satisfying the following conditions:
\begin{itemize}
\item[(1)] $L^2>8$. \item[(2)] $(L\cdot C_{x})\geq 4$ for all
irreducible curves $C_{x}$ passing through any very general point
$x \in S$.
\end{itemize}
Then the linear system $|K_S+\roundup{L}|$ separates two distinct points in very general positions. Consequently, $|K_S+\roundup{L}|$ gives a birational map.
\end{lem}

\begin{lem}\label{>1} Let $\sigma:S\longrightarrow S_0$ be a birational contraction from a nonsingular projective surface $S$ of general type onto the minimal model $S_0$. Assume that $(K_{S_0}^2, p_g(S_0))\neq (1,2)$ and that $C$ is a moving  curve on $S$. Then $(\sigma^*(K_{S_0})\cdot C)\geq 2$.
\end{lem}
\begin{proof} When $K_{S_0}^2\geq 2$, this is due to Hodge index theorem. When $(K_{S_0}^2, p_g(S_0))=(1,0)$, this is due to Miyaoka \cite[Lemma 5]{Mi}. When $(K_{S_0}^2,
p_g(S_0))=(1,1)$, $(\sigma^*(K_{S_0})\cdot C)=1$ implies
$K_{S_0}\equiv \sigma_* C$ by Hodge index theorem.  According to
Bombieri \cite{Bom}, we know that $S_0$ is simply connected. Thus
$K_{S_0}\sim \sigma_*C$, which is impossible since $|K_{S_0}|$ is
not movable.
\end{proof}

\begin{lem}\label{verygeneral} Let $\sigma:S\longrightarrow S_0$ be the birational contraction onto the minimal model $S_0$ from a nonsingular projective surface $S$ of general type. Assume that $(K_{S_0}^2, p_g(S_0))\neq (1,2)$ and that $\tilde{C}$ is a curve on $S$ passing through very general points. Then $(\sigma^*(K_{S_0})\cdot \tilde{C})\geq 2$.
\end{lem}
\begin{proof} In fact, by the projection formula, this is equivalent to see $(K_{S_0}\cdot C_0)\geq 2$ for any curve $C_0\subset S_0$ passing through very general points of $S_0$.

On the contrary, let us assume $(K_{S_0}\cdot C_0)\leq 1$. Then $g(C_0)\geq 2$ implies $C_0^2\geq 1$. The Hodge index theorem says $K_{S_0}^2=1$ and $K_{S_0}\equiv C_0$. Recall that $S_0$ is not a (1,2) surface. So $S_0$ must be either a $(1,0)$ surface or a $(1,1)$ surface.

If $(K_{S_0}^2, p_g(S_0))=(1,0)$, then $q(S_0)=0$ and the torsion element $\theta:=K_{S_0}-C_0$ is of order $\leq 5$ (see Reid \cite{Geaud}) and $h^0(S_0, C_0)=1$. Thus there are at most finite number of such curves on $S_0$ since $\#\text{Tor}(S_0)\leq 5$, which is absurd by the choice of $C_0$.

If $(K_{S_0}^2, p_g(S_0))=(1,1)$, then $q(S_0)=0$ and $K_{S_0}\sim C_0$ since $\text{Tor}(S_0)=0$ by Bombieri \cite[Theorem 15]{Bom} and thus $C_0$ is the unique canonical curve of $S_0$, which is absurd as well. \end{proof}

\subsection{The birationality principle}


\begin{defn} Pick two different generic irreducible elements $S'$, $S''$ (resp. $C'$, $C''$) in $|M_{m_0}|$ (resp. in $|G|$).
\begin{enumerate}
\item We say that $|mK_{X'}|$ { \it distinguishes $S'$ and $S''$}
if $\Phi_{|mK_{X'}|}(S') \neq \Phi_{|mK_{X'}|}(S'')$.

\item We say that $|mK_{X'}|$ {\it distinguishes $C'$ and $C''$}
if $\Phi_{|mK_{X'}|}(C')\neq \Phi_{|mK_{X'}|}(C'')$.
\end{enumerate}
\end{defn}

We will apply the useful, but technical theorem in Chen-Zuo
\cite{Chen-Zuo} for the birationality of $\Phi_m$.

\begin{thm}\label{birat} (see Chen-Zuo \cite[Theorem 3.1]{Chen-Zuo} or \cite[Theorem 2.11, Part 2]{Explicit2}) Keep the same notations as above. Assume that, for some $m>0$, $|mK_{X'}|$ distinguishes $S'$ and $S''$, $C'$ and $C''$ for generic $S'\neq S''$, $C'\neq C''$. Then $\Phi_m$ is birational under one of the following conditions:
\begin{itemize}
\item[(i)] $\alpha_m>2$; \item[(ii)] $\alpha_m>1$ and $C$ is not
hyper-elliptic.
\end{itemize}
\end{thm}
\section{\bf The lower bound of $K^3$ in terms of $m_0$} 

In the study of   3-dimensional explicit birational geometry, a
challenging problem is to determine whether a given weighted basket
${\Bbb B}$ is geometric, i.e. equal to ${\Bbb B_X}$ for some 3-fold $X$ or not. By exploiting geometric
properties, one might be able to have a better estimation of the
lower bound of $K_X^3$,  and hence exclude some non-geometric
formal baskets.
In fact, in \cite[2.19$\sim$2.31]{Explicit2}, we already proved some effective inequalities for $K_X^3$.  
We shall go further along this direction in this section

 Let $X$ be a minimal 3-fold of general type. Assume $P_{m_0}(X)\geq 2$. Mostly we will take $\Lambda=|m_0K_X|$.
Keep the settings in \ref{setup} and \ref{3fold}.

\subsection{}\label{d=3} {\bf The case} $d_{m_0}=3$.\\
 If we take $|G|$ to be $|S|_S|$, then $\beta=\frac{1}{m_0}$. It is known, from \cite[2.19]{Explicit2},  that $\deg(K_C)\geq 6$, $\xi\geq \frac{10}{3m_0+2}$ and $K_X^3\geq \frac{\xi}{m_0^2}$.
 Take $m= 5m_0+4,\cdots, (2t+1)m_0+2t$, successively. Then, by (\ref{i3}),
 one has $\xi \ge  \frac{17}{5m_0+4},  \frac{24}{7m_0+6}$, $\cdots,$ $\frac{7t+3}{(2t+1)m_0+2t}$ respectively.
 Taking the limit, we obtain $\xi \geq \frac{7}{2m_0+2}$. Therefore
\begin{equation}\label{j0} K_X^3 \ge \frac{7}{2m_0^2(m_0+1)}. \end{equation}

 In fact, for each small $m_0$, the explicit lower bound of $K^3$ can be slightly improved by the same trick and here is the result:
\medskip

{\small \centerline{\underline{Table A1}}
\smallskip

\begin{center}
\begin{tabular}{|c|c|c|c|c|c|c|c|}
\hline
 $m_0=$ & 2 & 3 & 4& 5 &6& 7& 8 \\
\hline $\xi\geq $& 4/3 &1&3/4&5/8&1/2
&6/13 & 2/5\\
\hline
$K^3\geq$&1/3 & 1/9& 3/64&1/40 & 1/72&6/637 &1/160\\
\hline\hline
 $m_0=$  & 9 & 10 & 11&12& 13& 14&15\\
\hline $\xi\geq $ &4/11&1/3&3/10&5/18
&1/4 & 6/25 & 2/9\\
\hline
$K^3\geq$ &4/891&1/300&3/1210&5/2592 &1/696 & 3/2450 & 2/2025\\
\hline
\end{tabular}
\end{center}}



\subsection{}\label{d=2} {\bf The case} $d_{m_0}=2$.\\
If we take $|G|=|S|_S|$, then $\beta\geq \frac{P_{m_0}-2}{m_0}$.
By Inequality (\ref{i3}), one has $\xi\geq \frac{2}{2m_0+1}$. Take
$m=3m_0+2, 5m_0+4,\cdots, (2t+1)m_0+2t$ successively. One gets
from Inequality (\ref{i3}) that $\xi \ge \frac{4}{3m_0+2},
\frac{7}{5m_0+4}, \cdots,\frac{3t+1}{(2t+1)m_0+2t}$. Taking the
limit, we have $\xi \geq \frac{3}{2m_0+2}$. By Inequality
(\ref{i1}), we have
\begin{equation}\label{j1}K_X^3 \geq \frac{3(P_{m_0}-2)}{2m_0^2(m_0+1)}\geq\frac{3}{2m_0^2(m_0+1)}.\end{equation}
In fact, we have the following  estimation for each small $m_0$,
which slightly improves \cite[Table A]{Explicit2}:
\medskip

{\small \centerline{\underline{Table A2}}
\smallskip

\begin{center}
\begin{tabular}{|c|c|c|c|c|c|c|c|}
\hline
 $m_0=$ & 2 & 3 & 4& 5 &6& 7 &8\\
\hline $\xi\geq $& 1/2 &2/5&1/3&1/4&2/9
&1/5 &1/6\\
\hline
$K^3\geq$&1/8 & 2/45&1/48 &1/100&1/162&1/245 & 1/384\\
\hline\hline
 $m_0=$  & 9 & 10 & 11&12& 13 & 14 & 15\\
\hline $\xi\geq $&2/13&1/7&1/8&2/17
&1/9 & 1/10 & 2/21\\
\hline
$K^3\geq$ &2/1053&1/700&1/968&1/1224& 1/1521 & 1/1960 & 2/4725\\
\hline
\end{tabular}
\end{center}}
\medskip

Under the same situation, if there exists a number $m_1>0$ such
that $d_{m_1}=3$, then, since $(m_1\pi^*(K_X)|_F\cdot C)\geq 2$,
we have $\xi\geq \frac{2}{m_1}$. Thus Inequality (\ref{i1}) reads:
\begin{equation}\label{j2}
K_X^3\geq \frac{2(P_{m_0}-2)}{m_0^2m_1}\geq \frac{2}{m_0^2m_1}.
\end{equation}

\subsection{}\label{>0} {\bf The case} $d_{m_0} =1$,
$b=g(\Gamma)>0$.\\
We have $S=F$ by definition. Pick a very large number $l>0$. Take
$|G|:=|l\sigma^*(K_{F_0})|$ which is base point free by the
surface theory. By definition, we have $\theta\geq P_{m_0}\geq 2$.
Since $\pi^*(K_X)|_F\sim \sigma^*(K_{F_0})$ by Lemma
\ref{cr}(i), we see $\beta=\frac{1}{l}$ and thus Inequality
(\ref{i1}) implies
\begin{equation}\label{j3}
K_X^3 \geq \frac{P_{m_0}}{m_0}\cdot \frac{1}{l}\cdot lK_{F_0}^2\ge \frac{P_{m_0}}{m_0}.\end{equation}

\subsection{} \label{=0} {\bf The case} $d_{m_0} =1$, $b=0$.\\
By Lemma \ref{cr}(ii), we have
\begin{equation}\label{j4}
K_X^3 \geq \frac{\theta}{m_0}\pi^*(K_X)|_F^2\geq
\frac{\theta^3}{m_0(m_0+\theta)^2}\cdot K_{F_0}^2.\end{equation}

We  will choose suitable linear system $|G|$ on $F$  depending on
the numerical type of $F$.  {}From the surface theory, we know
that either $K_{F_0}^2\geq 2$ or $(K_{F_0}^2, p_g(F))=(1,2)$,
$(1,1)$, $(1,0)$.
\medskip

\noindent {\bf Subcase \ref{=0}.1.} $K_{F_0}^2 \ge 2$.\\
Inequality (\ref{j4}) implies
\begin{equation}\label{j5}
K_X^3\geq \frac{2\theta^3}{m_0(m_0+\theta)^2}.\end{equation}
\medskip

\noindent {\bf Subase \ref{=0}.2. $(K_{F_0}^2,
p_g(F_0))=(1,2)$}.\\
Take $|G|:=\textrm{Mov}|K_F|$. Then $C$, as a generic irreducible
element of $|G|$, is a smooth curve of genus 2 (see \cite{BPV}).
By Lemma \ref{cr}(ii), we have
$\beta=\frac{\theta}{m_0+\theta}\geq \frac{1}{m_0+1}$.

Inequality (\ref{i2}) implies $\xi\geq \frac{\theta}{m_0+\theta}$.
Take $m=\rounddown{\frac{3m_0+3\theta}{\theta}}+1
> \frac{3m_0+3\theta}{\theta}$. Then, since $\alpha_m\geq
(m-1-\frac{m_0}{\theta}-\frac{1}{\beta})\xi> 1$, Inequality
(\ref{i3}) gives $\xi\geq
\frac{4}{\rounddown{\frac{3m_0+3\theta}{\theta}}+1} \ge
\frac{4\theta}{3m_0+4\theta}$. Inductively, take
$m=\rounddown{\frac{(1+\frac{2}{3}(4^t-1))m_0+3\cdot
4^{t-1}\theta}{4^{t-1}\theta}}+1$, one gets $\xi \ge \frac{4^t
\theta}{(1+\frac{2}{3}(4^t-1))m_0+ 4^{t}\theta}$ and hence $\xi
\ge \frac{3\theta}{2m_0+3\theta}$ by taking the limit.
 Thus we have
\begin{equation}\label{i10}K_X^3\geq \frac{3\theta^3}{m_0(m_0+\theta)(2m_0+3\theta)}\ge \frac{3}{m_0(m_0+1)(2m_0+3)}.\end{equation}

A similar calculation leads to  the following better estimation
for smaller $m_0$:
\smallskip

\centerline{\underline{Table A3}}
\smallskip

{\small \begin{center}
\begin{tabular}{|c|c|c|c|c|c|c|c|}
\hline
 $m_0=$ & 2 & 3 & 4& 5 &6& 7&8\\
\hline $\xi\geq $& 1/2 &1/3&2/7&1/4&1/5
&2/11&1/6\\
\hline
$K^3\geq$&1/12 & 1/36&1/70&1/120&1/210&1/308&1/432 \\
\hline\hline
 $m_0=$ & 9& 10 & 11 & 12&13&14&15\\
\hline $\xi\geq $&1/7&2/15&1/8&1/9&2/19
&1/10& 1/11\\
\hline
$K^3\geq$ &1/630&1/825&1/1056&1/1404&1/1729&1/2100&1/2640\\
\hline
\end{tabular}
\end{center}}
\medskip

\noindent {\bf Subcase \ref{=0}.3.  $(K_{F_0}^2,
p_g(F_0))=(1,1)$}. \\
Since $|\sigma^*(K_{F_0})|$ is not moving, we have to take
$|G|:=|2\sigma^*(K_{F_0})|$ which is base point free by the
surface theory.  Naturally the generic irreducible element $C$ of $|G|$ is
even and $\deg(K_C)=6$.

By Lemma \ref{cr}(ii), we have
$\beta=\frac{\theta}{2m_0+2\theta}$. Take
$m=\rounddown{\frac{3m_0+3\theta}{\theta}}+1$. Since $\xi>0$, we
have $\alpha_m>0$. Thus Lemma \ref{eveni} implies $\xi\geq
\frac{8\theta}{3m_0+4\theta}$.  Thus Inequality (\ref{i1}) reads
\begin{equation}\label{i11}K_X^3 \geq \frac{4\theta^3}{m_0(m_0+\theta)(3m_0+4\theta)}.\end{equation}

For each small $m_0$, we have the following better estimation:
\smallskip

\centerline{\underline{Table A4}}
\smallskip

\begin{center}{\small
\begin{tabular}{|c|c|c|c|c|c|c|c|}
\hline
 $m_0=$ & 2 & 3 & 4& 5 &6& 7&8\\
\hline
$\xi\geq $&  6/7&2/3&1/2&4/9&3/8&1/3&2/7\\
\hline
$K^3\geq$&1/14 &1/36 &1/80&1/135&1/224&1/336& 1/504\\
\hline\hline
 $m_0=$ & 9& 10 & 11 & 12&13&14&15\\
\hline $\xi\geq $&4/15&6/25&2/9&1/5&4/21
&14/79& 1/6\\
\hline
$K^3\geq$ &1/675&3/2750&1/1188&1/1560&1/1911&1/2370&1/2880\\
\hline
\end{tabular}}
\end{center}
\medskip

\noindent {\bf Subcase \ref{=0}.4. $(K_{F_0}^2, p_g(F_0))=(1,0)$}.\\
Modulo further birational modification, we may assume that
$\textrm{Mov}|2K_F|$ is base point free. Take
$|G|=\textrm{Mov}|2K_F|$. By Catanese-Pignatelli \cite{CP}, the
generic irreducible element $C$ of $|G|$ is a smooth curve of
genus $\geq 3$. By Lemma \ref{cr}(ii), we have
$\beta=\frac{\theta}{2m_0+2\theta}\geq \frac{1}{2m_0+2}$. Lemma
\ref{>1} implies $\xi\geq \frac{\theta}{m_0+\theta}\cdot
(\sigma^*(K_{F_0})\cdot C)\geq \frac{2\theta}{m_0+\theta}$.
Thus we have
\begin{equation}\label{i100}
K_X^3\geq \frac{\theta^3}{m_0(m_0+\theta)^2}.
\end{equation}
Of course, for each small $m_0$, one might get slightly better
estimation for $\xi$ and $K_X^3$.
\medskip

\noindent {\bf Variant \ref{=0}.5.} If there exists a positive
integer $m_1$ such that $P_{m_1}\geq 2$ and that $|m_0K_{X'}|$ and
$|m_1K_{X'}|$ are not composed with the same pencil. We may take
$|G|=|M_{m_1}|_F|$ and then we have $\beta=\frac{1}{m_1}$. Thus
Inequality (\ref{i1}) and Lemma \ref{>1} imply
\begin{equation}\label{j6}
K_X^3\geq \frac{2\ttm^2}{m_0m_1(m_0+\ttm)},\end{equation} provided
that $(K^2_{F_0}, p_g(F_0))\neq (1,2)$.
\medskip

\subsection{Some other inequalities}

\begin{cor}\label{m01} Let $X$ be  a minimal 3-fold of general type. Assume  $P_{m_0}=2$.  Keep the same notation as above.  Suppose that the general fiber $F$ of the induced fibration from $\Phi_{m_0}$ is not a $(1,2)$ surface, and that $P_{m_1} \geq 2$ for some integer $m_1>0$.
Then $$K_X^3 \ge \min \{ \frac{(P_{m_1}-1)^3}{m_1 (m_1+P_{m_1}-1)^2}, \frac{2}{m_0 m_1 (m_0+1)} \}.$$
\end{cor}

\begin{proof}
If $|m_0K_{X'}|$, $|m_1K_{X'}|$ are composed with the same pencil,
then both $|m_0K_{X'}|$ and $|m_1K_{X'}|$ induce the same
fibration $f:X'\lrw \Gamma$. Consider $\tilde{\Lambda}=|{m}_1
K_{X'}|$. Then, ${\theta}_{m_1} \ge P_{m_1}-1$. Since $F$ is not a
(1,2) surface and by comparing Inequality \ref{j3}, \ref{j5},
\ref{i11} and \ref{i100}, we have $$K_X^3
\ge\frac{(P_{m_1}-1)^3}{m_1 (m_1+P_{m_1}-1)^2}.$$

Suppose that $|m_0K_{X'}|$, $|m_1K_{X'}|$ are not composed with the same pencil.  We have $\beta=\frac{1}{m_1}$. Then we have Inequality (\ref{j6}) as in Variant \ref{=0}.5.
\end{proof}


Now we are able to study the more restricted case:

\begin{prop}\label{d24} Let $X$ be a minimal 3-fold of general type. Assume that $P_{m_0}(X)\geq 4$ and $d_{m_0}=2$, then

$$K_X^3\geq \min \mathrm{}\{\frac{8}{m_0(m_0+2)^2}, \frac{6}{m_0^2(m_0+2)}\}.$$
\end{prop}

\begin{proof} We need to study the image surface $W'$ of $X'$ through the morphism $\Phi_{|m_0K_{X'}|}$. In fact, we have the
Stein factorization $$\Phi_{m_0}:=\Phi_{|m_0K_{X'}|}: X'
\overset{f} \longrightarrow \Gamma \overset{s}\longrightarrow
W'\subset \mathbb{P}^{P_{m_0}-1}.$$ Denote by  $H'$ a very ample
divisor on $W'$ such that $M_{m_0}\sim\Phi_{m_0}^*(H')$.
Furthermore one has $M_{m_0}|_{S}\equiv \tilde{a}_{m_0}C$ for a
general member $S\in |M_{m_0}|$ and the integer
$\tilde{a}_{m_0}\ge \deg(s)\deg(W')\ge \deg(W')\ge P_{m_0}-2$,
where $C$ is a general fiber of $f$.  Set
$|G|:=|M_{m_0}|_S|$.
\medskip

\noindent {\bf Case 1.}  $\tilde{a}_{m_0}\geq 3$.\\
 We have $\beta\geq \frac{3}{m_0}$. Inequality (\ref{i2}) implies $\xi\geq \frac{6}{4m_0+3}$. Take $m=2m_0+2$. Then Inequality (\ref{i3}) gives $\xi\geq \frac{2}{m_0+1}$.
Take $m=\rounddown{\frac{11m_0+9}{6}}+1$. Since $\alpha_m>(\frac{11m_0+9}{6}-1-m_0-\frac{1}{\beta})\xi\geq 1$, Inequality (\ref{i3}) implies $\xi\geq \frac{24}{11m_0+15}$. Thus, we have \begin{equation}\label{k1}K_X^3\geq \frac{72}{m_0^2(11m_0+15)}.\end{equation}

\noindent {\bf Case 2.} $\tilde{a}_{m_0}=2$.\\
Automatically we have $P_{m_0}=4$, which also implies that
$\deg(W')=2$ and $\deg(s)=1$. Recall that an irreducible surface
(in $\mathbb{P}^3$) of degree 2  is one of the following surfaces
(see, for instance, Reid \cite[p. 30, Ex. 19]{Park}):
\begin{itemize}

\item[{\bf (a)}] $W'$ is the cone $\overline{\mathbb{F}}_2$ obtained by blowing
down the unique section with the self-intersection $(-2)$ on the
Hirzebruch ruled surface $\mathbb{F}_2$;

\item[{\bf (b)}] $W'\cong\mathbb{P}^1\times \mathbb{P}^1$.
\end{itemize}
\medskip

\noindent {\bf Case 2.a}. $W'=\overline{\mathbb{F}}_2$.\\

Replacing by its birational model, we may assume that $\Phi_{m_0}$
factors through the minimal resolution $\mathbb{F}_2$ of $W'$. So
we have the factorization of
$\Phi_{m_0}:X'\overset{h}\longrightarrow
\mathbb{F}_2\overset{\nu}\longrightarrow W'$ where $h$ is a
fibration and $\nu$ is the minimal resolution of $W'$. Set
$\hat{H}=\nu^*(H')$. We know that ${H'}^2=2$ and hence $\hat{H}^2
= 2$. Noting that $\hat{H}$ is nef and big on $\mathbb{F}_2$, we
can write
$$\hat{H}\sim \mu G_0+nT$$
where $\mu$ and $n$ are integers, $G_0$ denotes the unique section
with $G_0^2=-2$, and $T$ is the general fiber of the ruling on
$\mathbb{F}_2$. The property of $\hat{H}$ being nef and big
implies that $\mu>0$ and $n\ge 2\mu\ge 2$. Now let
$pr:\mathbb{F}_2\longrightarrow \mathbb{P}^1$ be the ruling.
Set $\tilde{f}:=pr\circ h: X'\longrightarrow \mathbb{P}^1$, which
is a fibration with connected fibers. Denote by ${F}$ a general
fiber of $\tilde{f}$.
 We have
$$M_{m_0}\sim \Phi_{m_0}^*(H')=h^*(\hat{H})\ge 2 {F}.$$
Let ${\Lambda}=|2 {F}|\lsleq |m_0K_{X'}|$.  Clearly we have
${\tl}=2$, $d_{{\Lambda}}=1$ and $b=0$. By Inequalities
(\ref{j5}), (\ref{i10}), (\ref{i11}) and (\ref{i100}),  we have
\begin{equation}\label{k2} K_X^3\geq \frac{8}{m_0(m_0+2)^2}.\end{equation}

\noindent {\bf Case 2.b}. $W'=\mathbb{P}^1\times \mathbb{P}^1$.\\
We have an induced fibration $f:X'\longrightarrow
W'=\mathbb{P}^1\times \mathbb{P}^1$. Since a very ample divisor
$H'$ on $W'$ with ${H'}^2=2$ is linearly equivalent to
$L_1+L_2=q_1^*(\text{point})+q_2^*(\text{point})$ where $q_1, q_2$
are projections from $\mathbb{P}^1\times \mathbb{P}^1$ to
$\mathbb{P}^1$ respectively. Set $\tilde{f_i}:=q_i\circ f:
X'\longrightarrow \mathbb{P}^1$, $i=1,2$. Then $\tilde{f_1}$ and
$\tilde{f_2}$ are two fibrations onto $\mathbb{P}^1$. Let $F_1$
and $F_2$ be general fibers of $\tilde{f_1}$ and $\tilde{f_2}$,
respectively. Then $F_1\cap F_2$ is simply a general fiber $C$ of
$f$. We will estimate $\xi$ in an alternative way. In fact, the
following argument is similar to the proof of \cite[Theorem
3.1]{Chen-Zuo}.

 Since $\tilde{a}_{m_0}=2$, we have $S|_S\sim 2C$. On the other hand, we have $S\geq F_1+F_2$. Modulo further birational modifications, we may write $m_0\pi^*(K_X)\equiv F_1+F_2+H_{m_0}'$ where $H_{m_0}'$ is an effective $\bQ$-divisor with simple normal crossing supports. For any integer $m>m_0+1$, we consider the linear system
 $$|K_{X'}+\roundup{(m-m_0-1)\pi^*(K_X)}+F_1+F_2|\lsleq |mK_{X'}|.$$
 Since $(m-m_0-1)\pi^*(K_X)+F_2$ is nef and big, Kawamata-Viehweg vanishing (\cite{KV, V}) gives the surjective map:
\begin{eqnarray*}
&&H^0(K_{X'}+\roundup{(m-m_0-1)\pi^*(K_X)}+F_2+F_1)\\
&\longrightarrow &
H^0(F_1, K_{F_1}+\roundup{(m-m_0-1)\pi^*(K_X)}|_{F_1}+C).
\end{eqnarray*}
Using the vanishing theorem again, one gets the surjective map:
\begin{eqnarray*}
H^0(F_1, K_{F_1}+\roundup{(m-m_0-1)\pi^*(K_X)|_{F_1}}+C)
\longrightarrow& H^0(C, K_C+\hat{D}_m)
\end{eqnarray*}
where $\hat{D}_m:=\roundup{(m-m_0-1)\pi^*(K_X)|_{F_1}}|_C$ with
$$\deg(\hat{D}_m)\geq
(m-m_0-1)\xi. $$
When $m$ is large enough so that $\deg(\hat{D}_m)\geq 2$, the above two surjective maps directly implies
\begin{equation}\label{i20}
m\xi\geq \deg(K_C)+\deg(\hat{D}_m)\geq 2+\roundup{(m-m_0-1)\xi}.
\end{equation}
In particular, we have $\xi\geq \frac{2}{m_0+1}$.

Take $m=2m_0+3$. Then $(m-m_0-1)\xi>2$ and Inequality (\ref{i20}) gives $\xi\geq \frac{5}{2m_0+3}$.

Assume $m_0>1$ and take $m=2m_0+2$. One gets $\xi\geq \frac{5}{2m_0+2}$.
Take $m =  \rounddown{\frac{7m_0+12}{5}}=\rounddown{\frac{7m_0+7}{5}} +1  > \frac{7m_0+7}{5}$, one has $\xi \ge \frac{4}{m} \ge \frac{20}{7m_0+12}$. Inductively, take $m = \rounddown{ \frac{(2+\frac{5}{3}(4^t-1)) m_0 + 2+\frac{10}{3} (4^t-1)}{5 \cdot 4^{t-1}} } $ for $t\geq 1$, one has $ \xi \ge \frac{5 \cdot 4^t}{(2+\frac{5}{3}(4^t-1)) m_0 + 2+\frac{10}{3} (4^t-1)}$. We have $\xi \ge \frac{3}{m_0+2}$ by taking the limit and hence
%
\begin{equation}\label{k4}K_X^3\geq \frac{1}{m_0}\cdot (\pi^*(K_X)|_S)^2\geq \frac{2}{m_0^2}\cdot \xi\geq
\frac{6}{m_0^2(m_0+2)}.\end{equation}
We conclude the statement by comparing \ref{k1}, \ref{k2} and \ref{k4}.
\end{proof}

\begin{cor}\label{pm=3} Let $X$ be a minimal 3-fold of general type. The following  holds:
$$K_X^3\geq
\begin{cases} \min \mathrm{}\{\frac{8}{m_0(m_0+2)^2}, \frac{7}{2m_0^2(m_0+1)}\},&\text{when}\ P_{m_0}\geq 4;\\
\frac{3}{2m_0^2(m_0+1)},
& \text{when}\ P_{m_0}=3.
\end{cases}$$
\end{cor}
\begin{proof} When $P_{m_0}\geq 4$,  $d_{m_0}=3, 2, 1$ and the inequality follows from comparing Inequality (\ref{j0}), Proposition \ref{d24}, Inequalities (\ref{j3}, \ref{j5}, \ref{i10}, \ref{i11}, \ref{i100})  (with $\ttm=3$), respectively.

When $P_{m_0}=3$, $d_{m_0}=2, 1$ and the inequality follows immediately by comparing Inequality (\ref{j1}) with  Inequalities (\ref{j3}, \ref{j5}, \ref{i10}, \ref{i11}, \ref{i100}) (with $\ttm=2$).
\end{proof}


\section{\bf Threefolds with $\delta(V)\leq 12$}

The purpose of this section is to prove the following sharper
bounds:

\begin{thm}\label{v(x)} Let $X$ be a minimal projective 3-fold of general type with $2\leq \delta(X)\leq 12$.
Then $K_X^3\geq v(\delta(X))$, where the function $v(x)$ is
defined as follows:
\medskip

\begin{center}
\begin{tabular}{|c|c|c|c|c|c|c|}
\hline
 $x$ & $2$ & $3$ & $4$ & $5$ &$6$ & $7$\\
\hline $v(x)$&${1}/{14}$ &${1}/{36}$&${1}/{90}$&${1}/{135}$&
${1}/{224}$&${1}/{336}$\\
\hline\hline
 $x$ & $8$ & $9$ & $10$ & $11$ &$12$ & $--$\\
\hline $v(x)$&${1}/{504}$ &${1}/{675}$&${3}/{2750}$&${1}/{1188}$&
${1}/{1560}$&$--$\\
\hline
\end{tabular}
\end{center}
\medskip
\end{thm}

We are going to estimate the lower bound of the volume, case by case, for a given $\delta$. The discussion here relies on those formulae in \cite[(3.6)-(3.12)]{Explicit1}.

\begin{prop}\label{delta=2} If $P_2(X)\geq 2$, then $K_X^3\geq \frac{1}{14}$. \end{prop}
\begin{proof} Set $m_0=2$.  By Table A1, Table A2, Inequalities (\ref{j3}) and (\ref{j5}), Table A3, Table A4 and Corollary \ref{pm=3}, we have $K_X^3\geq \frac{1}{14}$ unless $P_2=2$, $d_2=1$, $b=0$ and $F$ is of type $(1,0)$.

In the remaining case, we have that $\chi(\OO_X)=1$ by \cite[Lemma 2.32]{Explicit2}.
By \cite[Lemma 3.2]{Explicit2}, one has $P_4\geq 2P_2\geq 4$. If $d_4\geq 2$, then $K_X^3\geq \frac{1}{12}$ by Inequality (\ref{j6}) (with $m_0=2$, $m_1=4$, $\theta_2=1$).
 If $d_4=1$, then $|2K_{X'}|$ and $|4K_{X'}|$ are composed with the same pencil. Thus we have $K_X^3\geq \frac{27}{196}>\frac{1}{8}$ by Inequality (\ref{i100}) (with $m_0=4$, $\theta_4=3$).
\end{proof}

\begin{prop}\label{delta=3} If $P_3(X)\geq 2$, then
$K_X^3\geq \frac{1}{36}$.
\end{prop}

\begin{proof} Take $m_0=3$ and $\Lambda=|3K_{X'}|$.  One has $K_X^3\geq \frac{1}{36}$ by Table A1, Table A2, Inequalities (\ref{j3}), (\ref{j5}), Table A3, Table A4  and Corollary \ref{pm=3} ($m_0=3$) unless we are in Subcase \ref{=0}.4 with $P_3=2$. That is, $P_3=2, d_3=1, b=0$ and $F$ is of type $(1,0)$. Again,  $\chi(\OO_X)=1$.
Thus, for any $m\geq 2$, \cite[Lemma 3.2]{Explicit2} implies $P_{m+2} \geq P_m+P_2$.

By Corollary \ref{m01}, if $P_4 \ge 3$ (resp. $P_5 \ge 3$), then $K_X^3 \ge \frac{1}{24}$ (resp. $\frac{1}{30}$).
Suppose that  both $P_4\leq 2$ and $P_5 \le 2$, then $P_5=2$ and $P_2=0$.  By \cite[(3.6)]{Explicit1}, $n_{1,2}^0=5-8+P_4 < 0$, which is a contradiction. Hence either $P_4$ or $P_5 \ge 3$ in this case
 and we are done.
\end{proof}

\begin{prop}\label{delta=4} If $P_4(X)\geq 2$, then $K_X^3\geq \frac{1}{90}$.
\end{prop}

\begin{proof} Similarly, we have $K_X^3 \geq \frac{1}{80}$  unless $P_4=2$, $b=0$ and $F$ is of $(1,0)$ type.  In fact, in this situation, we have at least $K_X^3 \geq \frac{1}{100}$ by Inequality (\ref{i100}).  We will go a little bit further to investigate this situation.

\medskip
\noindent
{\bf 0.} We may and do assume that $P_2 \le 1$ and $P_3 \le 1$.

\noindent
{\bf 1.} If $P_7\geq 3$ (resp. $P_6 \ge 3$, $P_5 \ge 3$), then $K^3 \ge \frac{8}{567}>\frac{1}{80}$ (resp. $\frac{1}{60}, \frac{1}{50}$) by Corollary \ref{m01}(with $m_0=4$, and  $m_1=7,6,5$ respectively).  So we may assume $P_5, P_6, P_7 \leq 2$. Since $P_6\geq P_4+P_2$, we see that $P_2=0$ and $P_6=P_4=2$.



\noindent
{\bf 2.} If $P_3=0$, then $n_{1,3}^0= P_5-2 \geq 0$ implies $P_5=2$. Now
 $n_{1,4}^5=3-\sigma_5\geq 0$ gives $\sigma_5\leq 3$. However $n_{1,3}^5\geq 0$ implies $\sigma_5\geq 4$, a contradiction. {}We thus assume that $P_3 =1$ from now on.

\noindent
{\bf 3.} We thus can make the following complete table for $B^{(5)}$ depending on $P_5, \sigma_5$:

{\tiny
\begin{center}
\begin{tabular}{|c|c|c||c|c|c|c|}
\hline
No. &$P_5$ &  $\sigma_5$ & $B^{(5)}$ & $K^3$ & $\epsilon+P_7$ \\
\hline
1& 1 &  0 & $\{ 2 \times (1,2), (2,5), 5 \times (1,4)\}$ & $1/20$ & $4$ \\
\hline
2& 1 &  1 &  $\{3 \times (1,2),  (1,3),  4 \times (1,4), (1,r)\}$ &  $1/r-1/6$ & $4$ \\
\hline
3& 2 &  1 & $\{(1,2), 2 \times (2,5), 3 \times (1,4), (1,r)\}$ &  $1/r-3/20$ & $5$ \\
\hline
4& 2 &  2 &  $\{ 2 \times (1,2), (2,5), (1,3), 2 \times (1,4), (1,r_1), (1,r_2)\}$ &  $1/r_1+1/r_2-11/30$ & $5$ \\
\hline
5& 2 &  3 &  $\{3 \times (1,2),  2 \times (1,3),   (1,4), (1,r_1),(1,r_2), (1,r_3)\}$ &  $1/r_1+r_2+r_3-7/12$ & $5$ \\
\hline
\end{tabular}
\end{center}}

\medskip

\noindent
{\bf 4.} By definition, one has $\sigma_5 \le \epsilon \le 2 \sigma_5$.
Note that No. 1 is impossible because $\epsilon=0$ but $P_7 \le 2$ implies that $\epsilon \ge 2$, a contradiction.
In No. 3, $P_5=2$ implies  $P_7=2$ and hence $\epsilon=3 > 2 \sigma_5$, a contradiction.

In No. 2, one must have $P_7=2$ and $\epsilon=2 = 2 \sigma_5$. Hence $r \ge 6$. Then it follows that $K^3 \le K^3(B^{(5)})\le 0$, a contradiction.
Similarly, in No. 4, $ K^3(B^{(5)})> 0$ only when $r_1=r_2=5$. But then $\epsilon=2$, a contradiction.

\noindent
{\bf 5.} It remains to consider No. 5.
Note that $ K^3(B^{(5)})> 0$ only when $r_1=r_2=r_3=5$ and $K^3(B^{(5)})=\frac{1}{60}$.
There are only finitely many possible packings. Among them, we search for baskets with $K^3 \ge \frac{1}{100}$. It turns out there is only one new baskets
 $$B_{90}=\{3\times (1,2), 2\times (1,3), (2,9), 2\times (1,5)\}$$
with $K^3(B_{90})=\frac{1}{90}$.
\end{proof}

\begin{prop}\label{delta=5} If $P_5\geq 2$, then $K_X^3\geq \frac{1}{135}$.
\end{prop}
\begin{proof}  Similarly, we have $K_X^3\geq \frac{1}{135}$ unless $P_5=2$, $b=0$ and $F$ a $(1,0)$ surface, for which we have $K_X^3\geq \frac{1}{180}$.
 Furthermore,  we may assume that $P_m \le 2$ for $m=6,7,8$ by Corollary \ref{m01}.
It suffices to consider: $\chi(\OO_X)=1$, $P_2=0$, $P_3=0,1$,
$P_4=0,1$, $P_5=P_7=2$ and $ P_4 \le P_6 \le P_8 \le 2$.

We look at $B^{(5)}$ with $K^3 >0$ according to $(P_3, P_4, P_6)$ and $\sigma_5$.  It turns out that there is only one,
$$B^{(5)}=\{2\times (2,5), 3\times (1,3), (1,4), (1,6)\}$$ with $K^3(B^{(5)})=\frac{1}{60}$, given by $(P_3, P_4, P_6)=(1,1,2)$ and $\sigma_5=2$.
Now $P_8=2$ and hence
$$B^{(7)}=\{2\times (2,5), 2\times (1,3), (2,7), (1,6)\}.$$
However $K^3(B^{(7)}) = \frac{1}{210} < \frac{1}{180}$, which is impossible.
\end{proof}

\begin{prop}\label{delta=6} If $P_6\geq 2$, then $K_X^3\geq \frac{1}{224}$.
\end{prop}
\begin{proof} Similarly, we have $K_X^3\geq \frac{1}{224}$ unless $P_6=2$, $b=0$ and $F$ a $(1,0)$ surface,
for which we have $K_X^3\geq \frac{1}{294}$. Again, we may assume
that $P_m \le 2$ for $m=7, 8, 9,10$. Therefore, it remains to
consider such a situation that $\chi(\OO_X)=1$, $P_2=0$, $P_4\leq
1$, $P_3 \le P_5 \le 1$, $P_7 \le P_9 \le 2$ and $P_8=P_{10}=2$.
According to the value of $(P_3, P_4, P_5)$ and $\sigma_5$, we
have the following table.

 \medskip
{\tiny
\begin{center}
\begin{tabular}{|c|c|c||c|c|c|c|}
\hline
No. &$(P_3, P_4, P_5)$ &  $\sigma_5$ & $B^{(5)}$ & $K^3$ & $\epsilon+P_7$ \\
\hline
1& (0,0,0) &  0 & $\{ 5 \times (1,2), 4 \times (1,3), (1,4) \}$ & $1/12$ & $2$ \\
\hline
2& (0,0,1) &  0 &  $\{3 \times (1,2), 2*(2,5), 3* (1,3)\}$ &  $1/10$ & $3$ \\
\hline
3& (0,1,0) &  0 & $\{6*(1,2), (1,3), 3* (1,4)\}$ &  $1/12$ & $3$ \\
\hline
4& (0,1,1) &  0 & $\{4*(1,2),2*(2,5), 2* (1,4)\}$ &  $1/10$ & $4$ \\
\hline
5& (0,1,1) &  1 & $\{5*(1,2),1*(2,5), (1,3), (1,4), (1,r)\}$ &  $1/r-7/60$ & $4$ \\
\hline
6& (0,1,1) &  2 & $\{6*(1,2), 2*(1,3), (1,r_1), (1,r_2)\}$ &  $1/r_1+1/r_2-1/3$ & $4$ \\
\hline
7& (1,0,1) &  0 & $\{(2,5), 6*(1,3),(1,4)\}$ &  $1/20$ & $2$ \\
\hline
8& (1,0,1) &  1 & $\{(1,2), 7*(1,3),(1,r)\}$ &  $1/r-1/6$ & $2$ \\
\hline
9& (1,1,1) &  0 & $\{(1,2), (2,5), 3*(1,3),3*(1,4)\}$ &  $1/20$ & $3$ \\
\hline
10& (1,1,1) &  1 & $\{2*(1,2), 4*(1,3),2*(1,4), (1,r)\}$ &  $1/r-1/6$ & $3$ \\
\hline
\end{tabular}
\end{center}}

\medskip

\noindent{\bf 1.} It is clear that No. 2, 3, 4, 9 are not allowed
for $\epsilon=0$ and hence $P_7 \ge 3$.

\noindent{\bf 2.} In No. 1, 7, the baskets allow at most one
packing at level $7$, i,e, $\epsilon_7 \le 1$. However, $P_7=2$ and
$P_8=2$ yield $\epsilon_7 \ge 2$, a contradiction.

\noindent {\bf 3.} Consider No. 10. Since
$K^3=\frac{1}{r}-\frac{1}{6}>0$, it follows that $r=5$. So
$\epsilon=1$ and  $P_7=2$. Then $\epsilon_7=2$ and
$$B^{(7)}=\{2\times (1,2), 2\times (1,3), 2\times (2,7), (1,5)\}.$$
This already implies $\epsilon_8=0$ and so we get $P_9=3$, a
contradiction.

\noindent {\bf 4.}  Consider No. 8. Since $K^3>0$, thus we get
$$B^{(5)}=\{(1,2), 7\times (1,3), (1,5)\}.$$
 Since $B^{(5)}$ allows no further packing, hence
 $K_X^3=\frac{1}{30}$ in this case.

\noindent{\bf 5.} Consider No. 5. Since $K^3>0$, $r=6, 7, 8$. It
is easy to see that the basket with the smallest volume and
dominated by $B^{(5)}$ is
$$B_{210}=\{(7,15), (2,7), (1,6)\}$$
with $K^3=\frac{1}{210}$. Thus $K_X^3\geq \frac{1}{210}$.

\noindent {\bf 6.} Finally Consider No. 6. Since $K^3>0$,
$(r_1,r_2)=(5,5), (5,6), (5,7)$. It is easy to see that the basket
with the smallest volume and dominated by $B^{(5)}$ is
$$B_{105}=\{6\times (1,2), 2\times (1,3), (1,5), (1,7)\}$$
with $K^3=\frac{1}{105}$. Thus $K_X^3\geq \frac{1}{105}$.
\end{proof}

Note that, when $\delta(X)\geq 7$, we can  utilize our explicit classification in \cite[Section 3]{Explicit2}. We shall omit some details to avoid unnecessary redundancy.

\begin{prop}\label{delta=7} If $P_7\geq 2$, then $K_X^3\geq \frac{1}{336}$.
\end{prop}
\begin{proof} Similarly, we have $K_X^3\geq \frac{1}{336}$ unless
$P_7=2$, $b=0$,  $F$ a $(1,0)$ surface and $\chi(\OO_X)=1$.
Again, we may assume that $P_m\le 2$ for $m=8, 9$. Hence $P_9=2$
and $P_2=0$.

By $\epsilon_6=0$, we have $P_4+P_5+P_6=P_3+2 + \epsilon$. Hence
$(P_3,P_4,P_5,P_6)=(0,0,1,1), (0,1,0,1), (0,1,1,1)$ or $(1,1,1,1)$
which corresponds to Cases IV, V, VI, and VIII in \cite[Section
3]{Explicit2} respectively.  The classification implies that, if
$K_X^3<\frac{1}{336}$, then $B_X\succeq B_{\min}$, where
$B_{\min}$ is a minimal positive basket and belongs to one of the
following:

\begin{itemize}
\item[(b1)] $ B_{6,4}=\{(1,2), (6,13), (1,3), 2\times(1,5)\}$ with
$K^3(B_{6,4})=\frac{1}{390}$ and $P_9(B_{6,4})=3$;

\item[(b2)]  $B_{6,6}=\{3\times (1,2), (3,7), (2,5),(1,4),
(1,6)\}$ with $K^3(B_{6,6})=\frac{1}{420}$ and $P_9(B_{6,4})=3$;

\item[(b3)]  $B_{8,3}=\{2\times (2,5),(1,3), (3,11), (1,4)\}$ with
$K^3(B_{8,3})=\frac{1}{660}$.
 \end{itemize}

Clearly, Case b1 can not happen because $P_9(B_X) \ge
P_9(B_{\min})=3$.

In the Case b2, for the similar reason, $B_X\neq B_{6,6}$. Thus
$B_X\succeq B_{60}:=\{4\times (1,2), 2\times (2,5),(1,4), (1,6)\}$
and so $K_X^3\geq K^3(B_{60})=\frac{1}{60}$.

{}Finally, in Case b3, the proof of \cite[Theorem 3.11]{Explicit2}
implies that $B_X\neq B_{8,3}$ and $B_X\succeq B_{210}=\{2\times
(2,5),(1,3), (2,7), 2\times (1,4)\}$ with $K_X^3\geq
K^3(B_{210})=\frac{1}{210}$. We have proved the statement.
\end{proof}

It is now  immediately to see the following consequences:

\begin{cor}\label{<420} (=Corollary \ref{<<420}) Let $X$ be a minimal projective 3-fold of general type with $K_X^3<\frac{1}{336}$. Then $\delta(X)\geq 8$.
\end{cor}

\begin{prop}\label{delta=8-12}  Let $X$ be a minimal projective 3-fold of general type.
\begin{enumerate}
\item If $P_{8}\geq 2$, then $K_X^3\geq \frac{1}{504}$.
\item If $P_9\geq 2$, then $K_X^3\geq \frac{1}{675}$.
\item If $P_{10}\geq 2$, then $K_X^3\geq \frac{3}{2750}$.
\item If $P_{11}\geq 2$, then $K_X^3\geq \frac{1}{1188}$.
\item If $P_{12}\geq 2$, then $K_X^3\geq \frac{1}{1560}$
\end{enumerate}
\end{prop}
\begin{proof} We only prove (1). Other statements can be proved similarly.

When $P_8\geq 2$, Table A1, Table A2,  Inequalities (\ref{j3}). (\ref{j5}), Table A3  and Table A4 imply $K_X^3\geq \frac{1}{504}$ unless we are in Subcase \ref{=0}.4, for which  one has
$K_X^3\geq \frac{1}{420}$ by \cite[Theorem 1.2(2)]{Explicit2} since $\chi(\OO_X)=1$.
\end{proof}

Propositions \ref{delta=2}, \ref{delta=3}, \ref{delta=4}, \ref{delta=5}, \ref{delta=6}, \ref{delta=7} and \ref{delta=8-12} imply Theorem \ref{v(x)}.
\smallskip

An interesting by-product is the following:

\begin{cor}\label{pg=1}(=Corollary \ref{pg1}(1)) Let $X$ be a minimal projective 3-fold of general type with $p_g(X)=1$. Then $K_X^3\geq \frac{1}{75}$.
\end{cor}
\begin{proof}

We distinguish the following cases.

\noindent {\bf Case 1.}  $P_4\geq 3$. \\
By Corollary \ref{pm=3}, $K_X^3 \ge \frac{3}{160}$.

\noindent {\bf Case 2.} $P_4=2$. \\
   We have $K_X^3 \geq
\frac{1}{70}$  by Inequalities (\ref{j3}), (\ref{j5}) and Table A3
unless $b=0$ and $F$ is either a $(1,1)$ or a $(1,0)$ surface, for
which we necessarily have $h^2(\OO_X)=0$ and thus $\chi(\OO_X)=0$.
Reid's Riemann-Roch formula implies $P_5>P_4=2$. Now Corollary
\ref{m01}(with $m_0=4$, $m_1=5$) yields $K_X^3\geq \frac{1}{50}$.


\noindent {\bf Case 3.}  $P_4=1$.\\
Since $p_g(X)=1$, one has $P_m>0$ for all $m>1$. By
\cite[(3.10)]{Explicit1}, we have
$$P_4+P_5+P_6=3P_2+P_3+P_7+\epsilon\geq 3P_2+P_3+P_7.$$
If $P_4=1$ (which implies $P_3=P_2=1$), then we have
$$P_5\geq (P_7-P_6)+3\geq 3.$$

 Then, from \cite[(3.6)]{Explicit1},  $n_{1,4}^0\geq 0$ implies $\chi(\OO_X)\geq 3$. Due to our previous result \cite[Corollary 1.2]{JLMS} for irregular 3-folds, we may assume $q(X)=0$. Thus we have
$h^2(\OO_X)=\chi(\OO_X)\geq 3$. Take a sub-pencil $\Lambda$ of
$|5K_{X}|$. Then $\Lambda$ induces a fibration $f:X'\lrw \Gamma$
after Stein factorization. Let $F$ be the general fiber and $F_0$
be the minimal model of $F$.

\noindent{\bf Claim.} $K_{F_0}^2 \ge 2$.
\begin{proof} Clearly we may write
$$f_*\omega_{X'}=\OO_{\Gamma}\oplus \OO_{\Gamma}(e_2)\oplus\cdots\oplus
\OO_{\Gamma}(e_{p_g(F)-1})$$
with $-2\leq e_j\leq -1$ for all $j$, since $p_g(X')=1$.
Note that we have
\begin{eqnarray*}
h^2(\OO_X)&=&h^1(f_*\omega_{X'})+h^0(R^1f_*\omega_{X'})\\
&\leq &(p_g(F)-1)+h^0(R^1f_*\omega_{X'}). \end{eqnarray*} If
$q(F)>0$, we have $K_{F_0}^2\geq 2$ by the surface theory. If
$q(F)=0$, we have $R^1f_*\omega_{X'}=0$ and thus $p_g(F)\geq
h^2(\OO_X)+1\geq 4$. Hence we have $K_{F_0}^2\geq 4$ by the
Noether inequality.
\end{proof}
If $d_5\geq 2$, then  we may set $m_1=5$ and apply Inequality
(\ref{j6}), which  gives $K_X^3\geq \frac{1}{75}$.

If $d_5=1$, then $|5K_{X'}|$ and $\Lambda$ are composed with the
same pencil. Thus we have $\theta_5 \ge 2$ and Inequality
(\ref{j5}) gives $K_X^3\geq \frac{16}{245}$.
\end{proof}
\section{\bf Threefolds with $\delta(V)\geq 13$}


Let $X$ be a minimal projective 3-fold of general type with
$\delta(X) \ge {13}$.
Now we are in the natural position to classify baskets ${\mathbb B}(X)$ with $\delta(X)\geq 13$.  In fact, we have ${\mathbb B}^{12}\succeq {\mathbb B}(X)\succeq {\mathbb B}_{\textrm{min}}$ for certain minimal positive basket ${\mathbb B}_{\textrm{min}}$ listed in \cite[Table C]{Explicit2}, where ${\mathbb B}^{12}$ is also listed there.
However, as pointed
out in \cite[Proposition 4.5]{Explicit2}, our earlier classification in  \cite[Table C]{Explicit2} is not clean  since some minimal baskets in
Table C are actually known to be ``non-geometric''.

Recall that, by definition,  a geometric weighted basket is a basket of a
projective threefold of general type. Hence the
following properties hold:
\begin{itemize}
\item[A.] $P_{m}P_{n} \le P_{m+n}$  if  $P_{m}=1$ and $n>0$.
\item[B.] $P_m\geq 0$ for all $m>0$.
\item[C.] $K^3 \ge f(m_0)$ for some explicit function $f(x)$ given in Sections 3 and 4 provided that $P_{m_0} \ge 2$.
\end{itemize}

Indeed, if  ${\mathbb B}^{12}$  violates one of $A,B,C$,  then so does ${\mathbb B}(X)$. Therefore ${\mathbb B}(X)$ is non-geometric.
If ${\mathbb B}_{\textrm{min}}$ is non-geometric (e.g. cases No.  3a, 5b, 10a, $\cdots$, etc.), then we need to check all baskets between  ${\mathbb B}^{12}$ and ${\mathbb B}_{\textrm{min}}$.
The following Table H consists of non-geometric baskets with $\delta \ge 13$.
We keep the same notation as in Table C.




\medskip


\centerline{\bf \underline{Table H}}

\noindent{\tiny
$$
\begin{array}{|l|c|c|c|l|}
\hline
No. & ( P_{12},\cdots, P_{24}) &  (n_{1,2},n_{4,9},\cdots,n_{1,5}) \text{ or } B_{min} & K^3 & \text{Offending}\\
\hline
3a& ( 1,0,0,1,0,0,2,0,3,1,1,1,3 )  &\{(2,5),(3,8),*\} \succ \{(5,13),* \} & \frac{17}{30030}& P_8P_8>P_{16}\\

5b& ( 1,0,1,2,0,0,3,0,2,1,2,2,3 )  & \{(5,13),(4,15),* \} & \frac{1}{1170}& P_8P_8>P_{16}\\

8& (1,0,2,1,0,1,3,1,4, 3, 2, 2, 5 ) &  (7,1,0,1,0,2,0,0,6, 0,2,0,0,0,1)& \frac{1}{770}& P_6P_{10} >P_{16} \\

9& (1, 0, 2, -1,1,0,2,0,1,2,1, 0, 2 ) &  (9,0,0, 2,0,0, 1,1,4, 0,1,0,0, 1, 0) & \frac{1}{5544}&P_{15}=-1 \\

\end{array}$$}

\noindent{\tiny
$$
\begin{array}{|l|c|c|c|l|}
10a&(1,0,2,1,2,-1,2,0,2,2,1,2,4)&\{(4,9),(3,7),*\} \succ \{(7,16),* \} & \frac{1}{1680}& P_{17}=-1\\

11a&(1,0,2,0,2,0,2,2,2,1,1,1,3)&\{(3,8),(4,11),*\} \succ \{(7,19),* \} & \frac{1}{2660}&P_8P_{14}>P_{22}\\

13& (1, 0, 3, -1, 1, 1, 3, 1, 3, 3, 3, 1, 4
) &  (12,0,0, 2,0,2, 0,2,4, 0,2,0,0,1,0) & \frac{4}{3465}& P_{15}=-1\\

15a& (1,0,3,0,1,0,2,0,3,1,1,1,4) &\{(4,11),(1,3),*\} \succ \{(5,14),* \} & \frac{1}{2520}& P_8P_{14}>P_{22}\\
15b& (1,0,2,0,1,0,3,0,3,2,1,1,4) &\{(2,5),(3,8),*\} \succ \{(5,13),* \} & \frac{23}{36036}&P_8P_{14}>P_{22}\\

15c& (1,0,3,1,2,0,3,1,3,2,2,2,5)    & \{(7,16),(7,19),* \} & \frac{31}{31920}&P_8P_{14}>P_{22}\\

16c&(1,0,2,1,1,-1,3,-1,2,2,1,1,3) &\{ \{(5,13),(7,16)* \} & \frac{3}{16016} & P_{17}=-1\\

18a& (1, 0, 3, 0, 1, 0, 2, 1, 2, 2, 2, 1, 3)
 &\{(4,11),(1,3),*\} \succ \{(5,14),* \} & \frac{1}{3080}& P_6P_{11}>P_{17}\\

19& (1,0,2,0,1,1,3,0,2,2,2,1,3) & (8,0,1,1,0,1,0,1,5,0,1,0,0,1,0) & \frac{2}{3465}& P_9P_{14}>P_{23}\\

20a& (1, 0, 1, 1, 1, 0, 3, -1, 2, 1, 0, 1, 3)
 &\{(2,5),(3,8),*\} \succ \{(5,13),* \} & \frac{1}{16380}& P_{19}=-1\\

21a&(1, 1, 1, 1, 2, 0, 2, 1, 2, 1, 2, 2, 3)
  &\{(1,3),(3,10),*\} \succ \{(4,13),* \} & \frac{1}{4680}& P_8P_9>P_{17}\\

22&(1, 0, 1, 1, 1, 0, 2, 1, 3, 1, 1, 1, 3
 ) &  (7,1,0, 1,0,1, 1,0,5, 1,0,0,1,0,1) & \frac{1}{9240}&P_8P_9>P_{17}\\

23a&(1,0,2,1,2,0,2,1,3,1,2,2,3) &\{(4,9),(3,7),*\} \succ \{(7,16),* \} & \frac{1}{2640}& P_8P_{9}>P_{17}\\

24&(1, 0, 2, 0, 0, 1, 3, 0, 3, 2, 2, 0, 3
 ) &  (10,1,0, 1,0,3, 0,1,6, 0,2,0,0,1,0)& \frac{1}{3465}& P_8P_8>P_{16}\\

26a&(1,0,3,1,1,1,3,0,4,1,2,2,5 ) &\{(4,11),(1,3),*\} \succ \{(5,14),* \} & \frac{1}{1260}&P_9P_{10}>P_{19}\\

27.1&(1,0,2,2,1,1,5,0,4,3,3,3,6)&\{(2,5),(3,8),*\} \succ \{(5,13),* \}& \frac{71}{45045}&P_9P_{10}>P_{19}\\
27.2&(1,0,2,2,1,1,5,-1,3,2,2,2,4) &\{(2,5),(5,13),*\} \succ \{(7,18),* \} & \frac{1}{1386}&P_{19}=-1\\
27a&(1,0,2,2,1,1,5,-1,3,2,2,2,3) &\{(2,5),(7,18),*\} \succ \{(9,23),* \} & \frac{1}{1386}&P_{19}=-1\\

27b& (1,0,2,2,1,1,5,-1,3,2,2,2,5)& \{(5,13),(5,18),* \} & \frac{1}{1170}&
P_{19}=-1\\

29a&(1,1,3,1,2,2,2,1,3,1,2,2,3) &\{(5,14),(1,3),*\} \succ \{(6,17),* \} & \frac{1}{5335}&P_9P_{14}>P_{23} \\

32b&(1,0,3,1,1,1,3,1,3,2,3,2,4) &\{(4,11),(1,3),*\} \succ \{(5,14),* \} & \frac{1}{1386}&P_9P_{14}>P_{23}\\

33a&  (   1,1,2,0,2,1,1,1,2,2,1,2,3 )  & \{(3,10),(2,7),*\} \succ \{(5,17),* \} & \frac{1}{2856}& P_6P_{16}>P_{22}\\
34b& (1,1,2,0,1,1,3,0,3,3,1,2,4)   &\{(2,5),(3,8),*\} \succ \{(5,13),* \} & \frac{1}{1170}& P_6P_{13}>P_{19}\\

39a&  ( 1,1,2,1,3,0,2,1,3,2,2,3,4)  &\{(4,9),(3,7),*\} \succ \{(7,16),* \} & \frac{1}{1680}&P_6P_{16}>P_{22}\\
39b&  ( 1,1,2,1,3,1,2,1,3,2,2,3,5)  & \{(3,10),(2,7),*\} \succ \{(5,17),* \} & \frac{4}{5355}&P_6P_{16}>P_{22}\\

40.1&( 1,1,2,1,2,1,4,0,4,3,2,3,6)  &\{(2,5),(3,8),*\} \succ \{(5,13),* \} & \frac{41}{32760}&P_6P_{13}>P_{19}\\
40a&( 1,1,2,1,2,1,4,-1,3,2,1,2,4)  &\{(4,10),(3,8),*\} \succ \{(7,18),* \} & \frac{1}{2520}&P_6P_{13}>P_{19}\\
40b& ( 1,1,2,1,2,1,4,0,4,3,1,2,5)  &\{(2,5),(6,16),*\} \succ \{(8,21),* \} & \frac{1}{1260}&P_6P_{13}>P_{19}\\

43a&  (1,1,3,0,2,1,2,1,3,2,2,2,4)  & \{(4,11),(1,3),*\} \succ\{(5,14),* \} & \frac{1}{2520}&P_7P_{8}>P_{15}\\
43b&  (1,1,2,0,2,1,3,1,3,3,2,2,4)   &\{(2,5),(3,8),*\} \succ \{(5,13),* \} & \frac{23}{36036}&P_7P_{8}>P_{15}\\

44a&  ( 1,1,2,1,2,1,4,1,3,4,2,2,4)   & \{(2,5),(6,16),*\} \succ \{(8,21),* \} & \frac{1}{1386}&P_7P_{18}>P_{25}=3\\
44b&  (1, 1, 2, 1, 2, 0, 3, 0, 2, 3, 2, 2, 3
)   & \{(7,16),(5,13),* \} & \frac{3}{16016}& P_7P_{10}>P_{17}\\


46a& (1, 1, 1, 1, 2, 1, 3, 0, 3, 1, 1, 2, 3)  &\{(2,5),(3,8),*\} \succ \{(5,13),* \} & \frac{1}{16380}& P_9P_{10}>P_{19}\\

50a&( 1,1,3,1,2,2,3,1,4,2,3,3,5 )&\{(4,11),(1,3),*\} \succ \{(5,14),* \} & \frac{1}{1260}& P_7P_{14}>P_{21}\\

51a&  ( 1,1,2,2,2,2,5,0,3,3,3,3,4 )  &\{(4,10),(3,8),*\} \succ \{(7,18),* \} & \frac{1}{1386}&P_6P_{13}>P_{19}\\
51b&  ( 1,1,2,2,2,2,5,0,3,3,3,3,5 )   & \{(5,13),(5,18),* \} & \frac{1}{1170}&P_6P_{13}>P_{19}\\

52a&  (1,1,2,1,1,0,2,1,2,2,1,2,3) &\{(2,5),(3,8),*\} \succ \{(5,13),* \} & \frac{1}{2184}& P_5P_{12}>P_{17}\\

56a& ( 1,1,2,2,1,1,2,1,3,2,2,3,3 ) &\{(4,9),(3,7),*\} \succ \{(7,16),* \} & \frac{1}{1680}& P_5P_{14}>P_{19}\\
57&  ( 1,0,2,2,0,1,3,1,3,2,2,2,3 ) &  (3,0,1, 2,0,5, 0,0,4, 0,0,1,0, 0, 0) & \frac{1}{1386} & P_7P_{9}>P_{16}\\

58a&  (1,1,2,2,2,0,2,1,3,2,2,3,4)  &\{(4,9),(3,7),*\} \succ \{(7,16),* \} & \frac{1}{1680}& P_5P_{12}>P_{17}\\

59a&  (1,1,2,1,2,1,2,3,2,2,2,2,3) &\{(3,8),(4,11),*\} \succ \{(7,19),* \} & \frac{1}{2660}& \text{Item C}\\

60a& (1, 1, 1, 2, 1, 1, 3, 0, 3, 1, 1, 2, 3) &\{(2,5),(3,8),*\} \succ \{(5,13),* \} & \frac{1}{16380}& P_9P_{10}>P_{19}\\

61& (1, 1, 1, 2, 1, 1, 2, 2, 3, 2, 2, 2, 3
) & (0,1,0, 1,0, 3,1,0,2,0,0,0,1,0,0 )& \frac{1}{9240} & \text{Item C}\\

62a& (1,1,2,2,2,1,2,2,3,2,3,3,3 )  &\{(4,9),(3,7),*\} \succ \{(7,16),* \} & \frac{1}{2640}& \text{Item C}\\

63&(1, 1, 3, 1, 2, 1, 3, 2, 3, 3, 2, 2, 4
 ) &  (5,0,1, 2,0,1, 1,1,3, 0,1,0,0,0, 1) & \frac{1}{5544} & \text{Item C}\\
\hline
\end{array}$$}
\medskip



By eliminating non-geometric baskets, we obtain a shorter list of baskets, listed in Table F-0, F-1, F-2 in the Appendix.
We summarize some observations from the Tables.


\begin{thm}\label{-13} (=Theorem \ref{M1}) Let $X$ be a minimal projective 3-fold of general type with the weighted basket $\mathbb{B}(X):=\{B_X, P_2,\chi(\OO_X)\}$. If $\delta(X)\geq 13$, then $P_2=0$ and $\mathbb{B}(X)$ belongs to one of the types listed in Tables F--0$\sim$ F--2 in Appendix. Furthermore, the following holds:
\begin{itemize}
\item[(1)] $\delta(X)=18$ if and only if $\mathbb{B}(X)=\{B_{2a},
0, 2\}$ (see Table F--0 for $B_{2a}$) with $K_X^3=\frac{1}{1170}$.

\item[(2)] $\delta(X)\neq
16,17$.

\item[(3)] $\delta(X)=15$ if and only if ${\mathbb B}(X)$
is among one of the cases in Table F--1. One has $K_X^3\geq
\frac{1}{1386}$.

\item[(4)] $\delta(X)=14$ if and only if
${\mathbb B}(X)$ is among one of the cases in Table F--2. One has
$K_X^3\geq \frac{1}{1680}$.

\item[(5)] $\delta(X)=13$ if and only
if $\mathbb{B}(X)=\{B_{41}, 0, 2\}$ (see Table F--0 for $B_{41}$) with $K_X^3=\frac{1}{252}$.
\end{itemize}
\end{thm}

Theorem \ref{v(x)}, Theorem \ref{-13} and \cite[Theorem 1.4]{MA} imply the following:

\begin{cor}\label{V} (=Theorem \ref{M2}(2)) Let $X$ be a minimal projective 3-fold of general type. Then
$K_X^3\geq \frac{1}{1680}$, and equality holds  if and only if
$\chi(\OO_X)=2$, $P_2=0$ and $B_X=B_{7a}$ or $B_X=B_{36a}$ (cf.
Table F--2).
\end{cor}

Theorem \ref{-13}, together with  the explicit calculation,  also  implies the following:

\begin{cor}\label{effective} Let $X$ be a minimal projective 3-fold of general type.  Then,
\begin{itemize}
\item[(1)] if $\delta(X)=13$, $P_m>0$ for all $m\geq 10$;
\item[(2)] if $\delta(X)=14,15,18$,  $P_m>0$ for all $m\geq 20$.
\end{itemize}
\end{cor}


%
%
%
%
%
%
\section{\bf Birationality}

\begin{thm}\label{birat-18} Let $X$ be a minimal projective 3-fold of general type. If $\delta(X)=18$, then $\Phi_m$ is birational for all $m\geq 61$.
\end{thm}
\begin{proof} Set $m_0=18$. By Theorem \ref{-13}, we know that $B_X=B_{2a}$, $P_2=0$,  $\chi(\OO_X)=2$, $P_{19}=0$, $P_{24}=3$ and $K_X^3=\frac{1}{1170}$.  By \cite[Corollary 1.2]{JLMS}, we see $q(X)=0$. Thus $|18K_{X}|$ induces a fibration $f:X'\lrw \Gamma\cong \bP^1$.  We have $h^2(\OO_{X'})=h^2(\OO_X)=1$. Pick a general fiber $F$.  Since $P_{19}(X)=P_{19}({\mathbb B}_{2a})=0$, we have
$H^0(X', K_{X'}+F)=0$.
\medskip

\noindent{\bf Claim \ref{birat-18}.1}.  $p_g(F)=1$.
\begin{proof}
Since $\chi(\OO_{X'})>1$, we have $p_g(F)>0$ by \cite[Lemma 2.32]{Explicit2}. On the other hand, we have the long exact sequence:
$$H^0(X', K_{X'}+F)\lrw H^0(F, K_F)\lrw H^1(X', K_{X'})\lrw H^1(X', K_{X'}+F)$$
which implies $h^0(K_F)\leq h^1(X', K_{X'})=h^2(\OO_{X'})=1$. Thus we get $p_g(F)=1$.
\end{proof}

We have $P_m>0$ for all $m\geq 20$ by Corollary \ref{effective} (2).  Consider the linear systems
$$|K_{X'}+\roundup{n\pi^*(K_X)}+F|\lsleq |(n+19)K_{X'}|.$$
Clearly  $|(n+19)K_{X'}|$  distinguish different general fibers
$F$ as long as $n\geq 19$. By Kawamata-Viehweg vanishing,
\begin{eqnarray*}
|K_{X'}+\roundup{n\pi^*(K_X)}+F||_F&=&|K_F+\roundup{n\pi^*(K_X)}|_F|\\
&\lsgeq& |K_F+\roundup{L_n}|
\end{eqnarray*}
where we set $L_n:=n\pi^*(K_X)|_F$.
\medskip

\noindent{\bf Claim \ref{birat-18}.2.} $L_n^2>8$ whenever $n\geq 42$.
\begin{proof} Since $p_g(F)=1$, we are in Subcase \ref{=0}.1 or Subcase \ref{=0}.3.

Let us consider Subcase \ref{=0}.1 (i.e. $K_{F_0}^2\geq 2$) first. We have
$$(\pi^*(K_X)|_F)^2\geq \frac{1}{19^2}K_{F_0}^2\geq \frac{2}{19^2}$$ by Lemma \ref{cr}(ii).
Thus $L_n^2 >8$ whenever $n > 38$.

If $K_{F_0}^2=1$, we shall estimate $L_n^2$ in an alternative way.
Suppose that $|24K_{X'}|$ and $|18K_{X'}|$ are not composed with
the same pencil. Take $|G|:=|M_{24}|_F|$. Pick a generic
irreducible element $C$ of $|G|$. Then we have
$\xi=(\pi^*(K_X)|_F\cdot C)\geq \frac{2}{19}$ by Lemma \ref{>1}.
Thus $(\pi^*(K_X)|_F)^2\geq \frac{1}{24}\xi\geq \frac{1}{12\cdot
19}$. Since $r(X)=2340$ and $r(X)(\pi^*(K_X)|_F)^2$ is an integer, we see $(\pi^*(K_X)|_F)^2\geq
\frac{11}{2340}$. So we have $L_n^2>8$ whenever $n\geq 42$.

Assume that $|24K_{X'}|$ and $|18K_{X'}|$ are composed with the
same pencil. Since $P_{24}=3$, we may set ${m}_0=24$ and
${\Lambda}=|24K_{X'}|$.  We have ${\theta}=2$. The
argument in Subcase \ref{=0}.3  implies that
$$(\pi^*(K_X)|_F)^2\geq \frac{4{\theta}^2}{(\tilde{m}_0+{\theta})(3{m}_0+4{\theta})}=\frac{1}{130}.$$ We have $L_n^2>8$ whenever
$n\geq 33$.
\end{proof}

For very general curves $\tilde{C}$ on $F$, one has
$$(L_n\cdot \tilde{C})\geq \frac{n}{19}(\sigma^*(K_{F_0})\cdot \tilde{C})\geq \frac{2n}{19}$$
by Lemma \ref{verygeneral}. Therefore,  $(L_n\cdot \tilde{C})\geq
4$  for $n \ge 38$. Lemma \ref{masek} implies that
$|K_F+\roundup{L_n}|$ gives a birational map for $n\geq 42$. Thus
$\Phi_m$ is birational for all $m\geq 61$.
\end{proof}

\begin{thm}\label{birat-15}  Let $X$ be a minimal projective 3-fold of general type.  If $\delta(X)\leq 15$, then $\Phi_m$ is birational for all $m\geq 56$.
\end{thm}
\begin{proof} Set $m_0=\delta(X)$. By considering a sub-pencil $\Lambda$ of $|m_0K_X|$, we may always assume that we have an induced fibration $f:X'\lrw \Gamma$ onto a curve $\Gamma$.
By Chen-Hacon \cite{irreg}, we may assume $q(X)=0$. Thus $\Gamma\cong \bP^1$.
By \cite[Corollary 3.13]{Explicit2} and \cite[Lemma 2.32]{Explicit2}, we know that $\delta(X)\leq 10$ as long as $F$ is a $(1,0)$ surface. Therefore it suffices to consider the following 3 cases:
\begin{enumerate}
\item[{\bf 1.}] $\delta(X)\leq 15$ and $F$ is a $(1,2)$ surface.
\item[{\bf 2.}] $\delta(X)\leq 15$ and $F$ is neither a $(1,2)$ surface nor a $(1,0)$ surface.
\item[{\bf 3.}] $\delta(X)\leq 10$ and $F$ is a $(1,0)$ surface.
\end{enumerate}
\medskip

\noindent {\bf Case 1.} Without losing of generality, let us assume $\delta(X)=15$. Take $|G|$ to be the moving part of $|K_F|$. Then, by Table A3, we have $\xi\geq \frac{1}{11}$.
We have $m_0=15$ and $\beta\mapsto \frac{1}{16}$. So $\alpha_m>2$ whenever $m\geq 55$. By Corollary \ref{effective}$,|mK_{X'}|$ separates different general fibers $F$ as long as $m\geq 35$.
On the other hand, Kawamata-Viehweg vanishing and Lemma \ref{cr} imply the following, whenever $m\geq 49$,
\begin{eqnarray*}
|mK_{X'}||_F&\lsgeq&|K_{X'}+\roundup{(m-16)\pi^*(K_X)}+F||_F\\
&\lsgeq& |K_F+\roundup{(m-16)\pi^*(K_X)|_F}\\
&\lsgeq& |(K_F+\roundup{Q_m}+C)+C|
\end{eqnarray*}
where $Q_m$ is a nef and big $\bQ$-divisor. Thus, by \cite[Lemma
2.17]{Explicit2}, $\Phi_m$ distinguishes different generic curves
$C$ for $m\geq 49$. {}Finally Theorem \ref{birat} implies that
$\Phi_m$ is birational for all $m\geq 55$.
\medskip

\noindent{\bf  Case 2.} Still assume $\delta(X)=15$.  Parallel to the respective argument in the proof of Theorem \ref{birat-18},  one knows that $|mK_{X'}|$ distingishes different general fibers $F$ for $m\geq 35$. By the surface theory, we see that $F$ is either a surface with $K_{F_0}^2\geq 2$ or a $(1,1)$ surface. We want to study the linear system $|K_F+\roundup{L_n}|$. In fact, by the estimation in Subcase \ref{=0}.1 and Table A4, we have $L_n^2\geq \frac{n^2}{32\cdot 6}>8$ whenever $n\geq 40$. Similarly we have $(L_n\cdot \tilde{C})\geq 4$ for all $n\geq 32$ and for all curves $\tilde{C}$ on $F$ passing through very general points.
By Lemma \ref{masek}, we see that $|K_F+\roundup{L_n}|$ gives a birational map for all $n\geq 40$.
Similar to what discussed in the proof of Theorem \ref{birat-18},
we have proved that  $\Phi_m$ is birational for all $m\geq
n+16\geq 56$.
\medskip

\noindent {\bf Case 3.} When $\delta(X)\leq 10$, we have much
better birationality result even though $F$ is a $(1,0)$ surface.
In fact, parallel argument shows that $\Phi_m$ is birational for
all $m\geq 39$. The proof is more or less similar to above ones.
We leave it as an exercise to interested readers.
\end{proof}

Theorems \ref{-13}, \ref{birat-18}, and \ref{birat-15} imply
Theorem \ref{M2} (2).

\section{\bf Threefolds with $\delta(V)=2$}
 This section is devoted to classifying minimal projective threefolds of general type with $\delta(X)=2$, that is, $p_g(X) \le 1$ and $P_2(X) \ge 2$.

Assume that $P_2\geq 2$.  We first recall the following known results:
\begin{itemize}
\item[(a)] If $d_2=3$, then $\Phi_m$ is birational for all $m\geq
7$ by \cite[Theorem 2.20]{Explicit2}.

\item[(b)] If $d_2=2$, $\Phi_m$ is birational for all $m\geq 10$
by \cite[Theorem 2.22]{Explicit2}.

\item[(c)] If $q(X)>0$, then  $\Phi_m$ is birational for all $m\geq
7$ by Chen--Hacon \cite{irreg} and for $m=6$ by Chen-Chen-Jiang
\cite{5K}.
\end{itemize}

The purpose of this section is to prove that $\Phi_m$ is
birational for $m \ge 11$ and classify threefolds such that
$\Phi_{10}$ is not birational. Therefore, we may and do assume
that $q(X)=0$, $d_2=1$ and  $b=g(\Gamma)=0$. Let $F$ be the
general fiber of the induced fibration $f: X'\rightarrow \bP^1$
from $\Phi_2$.

\subsection{Birationality of $\Phi_m$  for $m\geq 11$ } 


\begin{lem}\label{R1}$|mK_{X'}|$ distinguishes different general fibers of $f$ for all $m\geq 9$.
\end{lem}
\begin{proof} When $p_g(F)>0$, by \cite[Proposition 2.15 (i)]{Explicit2},
one has $P_k>0$ for $k\geq 7$. Thus, for all $m\geq 9$, $mK_{X'}\geq
F$, hence $|mK_{X'}|$ distinguishes different general fibers of $f$.

When $p_g(F)=0$, one has $\chi(\OO_X) \le 1$ (cf. \cite[Lemma
2.32]{Explicit2}). By \cite[Lemma 3.2]{Explicit2}, one has
$P_5\geq P_2>0$. Then clearly $P_k>0$ for all $k\geq 5$. Thus, for
all $m\geq 7$, $mK_{X'}\geq F$ and hence $|mK_{X'}|$ distinguishes
different general fibers of $f$.
\end{proof}

\begin{prop}\label{b1} Assume $P_2(X)\geq 2$, $q(X)=0$, $d_2=1$ and $F$ is not a $(1,2)$ surface. Then $\Phi_m$ is birational for all $m\geq 10$.
\end{prop}
\begin{proof} Set $L_n:=n \pi^*(K_X)|_F$ which is a nef and big $\bQ$-divisor on $F$.  Kawamata-Viehweg vanishing gives the following surjective map:
\begin{eqnarray*}
H^0(X', K_{X'}+\roundup{n \pi^*(K_X)}+F) \longrightarrow H^0(F,
K_F+\roundup{n \pi^*(K_X)}|_F).
\end{eqnarray*}

Together with Lemma \ref{R1}, it is sufficient to prove that
$|K_F+\roundup{L_n}|$ gives a birational map for $n \ge 7$ because
$$|(n+3)K_{X'}|\ \lsgeq\ |K_{X'}+\roundup{n \pi^*(K_X)}+F|.$$

\noindent {\bf Claim \ref{b1}.1.} If $K_{F_0}^2 \ge 2$ or $F_0$ is
of type $(1,0)$,  then $|K_F+\roundup{L_n}|$ is birational for $n
\geq 7$.

First of all, for any curve $\tilde{C}\subset F$ passing through
very general points of $F$, we estimate $(L_n\cdot \tilde{C})$ for
$n\geq 7$. Clearly we have  $g(\tilde{C})\geq 2$. Set $m_0=2$ and
$\Lambda=|2K_{X'}|$. By  Lemma \ref{cr} and Lemma
\ref{verygeneral}, we have
$$(L_n \cdot \tilde{C})\geq 7(\pi^*(K_X)|_F\cdot \tilde{C})
\geq \frac{7}{3}(\sigma^*(K_{F_0})\cdot \tilde{C})>4.$$

If $K_{F_0}^2 \ge 2$, then we have $$L_n^2\geq 49
(\pi^*(K_X)|_F)^2\geq 49(\frac{1}{3}\sigma^*(K_{F_0}))^2\geq
\frac{98}{9}>8.$$

If $F_0$ is a $(1,0)$ surface, we have $P_4\geq 2P_2\geq 4$ since
$\chi(\OO_X)\leq 1$.  When $d_4\geq 2$,  we set $m_0=2$,
$\Lambda=|2K_{X'}|$ and $|G|=|M_4|_F|$. Then $ \beta=\frac{1}{4}$,
$\xi\geq \frac{1}{3}(\sigma^*(K_{F_0})\cdot C)\geq \frac{2}{3}$
and so $L_n^2\geq \frac{49}{6}>8$.

When $d_4=1$, we set $m_0=4$ and $\Lambda=|4K_{X'}|$. Clearly
$|2K_{X'}|$ and $|4K_{X'}|$ induce the same fibration $f$. Take
$|G|=|2\sigma^*(K_{F_0})|$.  Since $\theta \geq 3$, we have
$\beta\geq \frac{3}{14}$ by Lemma \ref{cr}. Thus $\xi\geq
\frac{6}{7}$ and so $L_n^2\geq 49\cdot \frac{3}{14}\cdot
\frac{6}{7} >8$. By Lemma \ref{masek}, the Claim follows.
\medskip

\noindent{\bf Claim \ref{b1}.2.} If $F_0$ is a $(1,1)$ surface,
then $|K_F+\roundup{L_n}|$ is birational for $n \geq 7$.

Following the similar argument as above, it is easy to see that
$L_n^2\geq \frac{64}{7}>8$ and $(L_n\cdot \tilde{C})\geq 4$ for
all $n \geq 8$.
We consider the linear system $|K_F+\roundup{7\pi^*(K_X)|_F}|$ in
an alternative way.  Note that $|2\sigma^*(K_{F_0}))|$ is base
point free.  Pick a generic irreducible element $C\in
|2\sigma^*(K_{F_0}))|$. Since $\OO_{\Gamma}(1)\hookrightarrow
f_*\omega_{X'}$,  we have $f_*\omega_{X'/\Gamma}^2\hookrightarrow
f_*\omega_{X'}^{10}$.  The semi-positivity implies that
$f_*\omega_{X'/\Gamma}^2$ is generated by global sections, which
directly implies $10K_{X'}|_F\geq C$.  Thus $\Phi_{10}$
distinguishes different $C$. By Lemma \ref{cr}, we have
$6\pi^*(K_X)|_F\equiv C+H_6$ for an effective $\bQ$-divisor $H_6$
on $F$. Thus  the vanishing theorem implies
$$ |K_F+\roundup{7\pi^*(K_X)|_F-H_6}||_C=|K_C+D|$$
with $\deg(D)\geq 2(\roundup{7\pi^*(K_X)|_F-C-H_6}\cdot
\sigma^*(K_{F_0}))\geq 2$. Since $C$ is non-hyperelliptic,
$|K_C+D|$ gives a birational map. Thus
$|K_F+\roundup{7\pi^*(K_X)|_F}|$  is birational.
\end{proof}

\begin{prop}\label{b2} Assume $P_2(X)\geq 2$, $q(X)=0$, $d_2=1$ and $F$ a $(1,2)$ surface. Then $\Phi_m$ is birational for all $m\geq 11$.
\end{prop}
\begin{proof} Take $|G|$ to be the moving part of $|\sigma^*(K_{F_0})|$. Modulo
birational modifications, we may assume that $|G|$ is base point free. Pick a generic irreducible element $C$ of $|G|$. It is also known that $g©=2$.
\medskip

\noindent{\bf Claim \ref{b2}.1} The linear system $|mK_{X'}|$ distinguishes different general members of $|G|$ for $m\geq 9$.

\begin{proof}
Clearly $|G|$ is composed with a rational pencil since $q(F)=0$.
We shall prove $|mK_{X'}|_{|F} \ \lsgeq\ |G|$ and thus the
statement  follows. In fact, by Lemma  \ref{cr},  we have
$$3\pi^*(K_X)\equiv \sigma^*(K_{F_0})+H_3$$
for an effective $\bQ$-divisor $H_3$ on $F$. Thus, for $m\geq 10$,
$$Q_m:=(m-3)\pi^*(K_X)_{|F}-2H_3-2\sigma^*(K_{F_0})\equiv (m-9)\pi^*(K_X)|_F$$
is nef and big. It follows that
$K_F+\roundup{Q_{m}}+\sigma^*(K_{F_0})>0$ by \cite[Lemma
2.14]{Explicit2}. We thus have the following:
 \begin{eqnarray*}
|mK_{X'}|_{|F}&\lsgeq&|K_{X'}+F+\roundup{(m-3)\pi^*({K_X})}|_{|F}\\
&=& |K_F+\roundup{(m-3)\pi^*({K_X})}_{|F}| \\
&\lsgeq& |K_F+\roundup{(m-3)\pi^*({K_X})_{|F}-2H_3}| \\
&=&|(K_F+\roundup{Q_{m}}+\sigma^*(K_{F_0}))+\sigma^*(K_{F_0})| \\
&\lsgeq&|\sigma^*(K_{F_0})|\ \lsgeq \ |G|
\end{eqnarray*}
where the first equality follows from the Kawamata-Viehweg vanishing
(\cite{KV,V}). Therefore,
$|mK_{X'}|$ distinguishes general members of $|G|$ for $m \ge 10$.
Moreover, for $m=9$, \begin{eqnarray*}
|9K_{X'}|_{|F} &\lsgeq &|5K_{X'}|_{|F}\
\lsgeq\
|K_{X'}+\roundup{2\pi^*(K_X)}+F|_{|F}\\
&=&
|K_F+\roundup{2\pi^*(K_X)}|_F|\ \lsgeq\ |G|
\end{eqnarray*}
where the equality is again due to Kawamata-Viehweg vanishing.  Hence $|9K_{X'}|$ distinguishes general members of $|G|$ as well, which asserts the claim.
\end{proof}
{}From Table A3, one has
 $\xi\geq \frac{1}{2}$. Take $m\geq
11$, then $\alpha_m\geq \frac{5}{2}>2$. This means that
$|mK_{X'}|_{|C}$ distinguishes points on $C$. Thus, by Theorem
\ref{birat} and Claim \ref{b2}.1, $\Phi_m$ is birational for all
$m\geq 11$.
\end{proof}

Now Theorem \ref{1/2}.1 follows from  Proposition \ref{b1} and
Proposition \ref{b2}. That is, if $P_2 \ge 2$, then $\Phi_m$ is
birational for $m \ge 11$.

If either $
\xi>\frac{1}{2}$ or $\beta>\frac{1}{3}$  then $\alpha_{10}>2$. Hence the following consequence is immediate.
\begin{cor}\label{cc1} Let $X$ be  a minimal projective 3-fold of general type. Assume $P_2(X)\geq 2$, $q(X)=0$, $d_2=1$ and $F_0$ a $(1,2)$ surface. If either $
\xi>\frac{1}{2}$ or $\beta>\frac{1}{3}$ or $P_2>2$, then
$\Phi_{10}$ is birational.
\end{cor}

Proposition \ref{b1}, Proposition \ref{b2} and Corollary \ref{cc1} also imply the following:

\begin{cor}\label{10}  Let $X$ be  a minimal projective 3-fold of general type. Assume $P_2\geq 2$ and $\Phi_{10}$ is not birational. Then $P_2=2$, $q(X)=0$ and $|2K_{X'}|$ is composed with a rational pencil of  $(1,2)$ surfaces.
\end{cor}

\subsection{Classification}
In the rest of this section, we  classify  minimal 3-folds $X$ of general type which satisfy the following assumptions:
\begin{quote} {\em
($\sharp$) $P_2(X)=2$ and $\Phi_{10}$ is not birational. }
\end{quote}
Note that Corollary \ref{10} implies that $|2K_X|$ induces a
fibration $f:X'\lrw \bP^1$ with the general fiber $F$ a $(1,2)$
surface.

\begin{lem}\label{l1} If $X$ satisfies $(\sharp)$, then $0\leq \chi(\OO_X)\leq 3$.
\end{lem}
\begin{proof}  Note that the general fiber $F$ of $f$  is a (1,2) surface. Since $q(F)=0$, we have $q(X)=0$, $h^2(\OO_X)=h^1(\bP^1, f_*\omega_{X'})$ and $p_g(X)=h^0(f_*\omega_{X'})$.
Since $P_2(X)=2$ implies $p_g(X)\leq 1$, we see $\chi(\OO_X)\geq  0$. By Fujita's semi-positivity(\cite{F}), we have $\chi(\OO_X)\leq 3$.
\end{proof}

\begin{thm}\label{10-} Let $X$ be a minimal projective 3-fold of general type. Assume $P_2=2$, $q(X)=0$ and $F$ a $(1,2)$ surface. Then $\Phi_{10}$ is birational under one of the following conditions:
\begin{itemize}
\item[(1)] $P_3\geq 4$;
\item[(2)] $P_4\geq 6$;
\item[(3)] $P_5\geq 8$;
\item[(4)] $P_6\geq 14$.
\end{itemize}
\end{thm}
\begin{proof} We  set $m_0=2$. Pick a general fiber $F$ of $f:X'\lrw \Gamma$ and a generic irreducible element $C$ of $|G|:=\text{Mov}|\sigma^*(K_{F_0})|$ on $F$.  For $m_1=3,4,5$ and $6$, we have $P_{m_1}\geq 4$. Modulo further birational modifications to $\pi$, we may assume that  the moving part $|M_{m_1}|$ of $|m_1K_{X'}|$ is base point free.
We consider the following natural maps:
$$H^0(X', S_{m_1})\overset{\mu_{m_1}}\lrw H^0(F, S_{m_1}|_F)\overset{\nu_{m_1}}\lrw H^0(C, S_{m_1}|_C)$$
where $S_{m_1}\in |M_{m_1}|$ denotes the general member.

 Let $ \text{Mov}|S_{m_1}|_F| $ be the moving part of $|S_{m_1}|_F|$ and let $T_{m_1}$ be a general element in $ \text{Mov}|S_{m_1}|_F| $ when $h^0(F, S_{m_1}|_F)>1$. Clearly
 $$(S_{m_1}\cdot C)_{X'}\geq (T_{m_1}\cdot C)_F\geq 0.$$
  Since $F$ and $C$ are general, both $\mu_{m_1}$ and $\nu_{m_1}$ are non-zero maps. In particular, $h^0(F,S_{m_1}|_F)>0$ and $h^0(C,S_{m_1}|_C)>0$.

Let $F_{( r)}$ be a general element in $\text{Mov} | S_{m_1}-rF|$ if $h^0(S_{m_1}-rF) \ge 2$.
Let $C_{( r)}$ be a general element in  $\text{Mov} | T_{m_1}-rC|$ if $h^0(T_{m_1}-rC) \ge 2$.
Replace $X'$ by its birational modification, we may and do assume that $\text{Mov} | S_{m_1}-rF|$ is free.

Clearly, for $0<r \le \frac{h^0(X', S_{m_1})}{h^0(F,
S_{m_1}|_F)}$, we have
\begin{equation} \label{mult1} \begin{array}{l} h^0(X', S_{m_1}-rF)
\ge h^0(X', S_{m_1}) - r \cdot h^0(F, S_{m_1}|_F). \\

\end{array}
\end{equation}

\noindent
\noindent{\bf Claim \ref{10-}.1.} {\em If $(T_{m_1}\cdot C)\leq 1$, then  $(T_{m_1}\cdot C)=0$.}
\begin{proof}
In fact, if $|T_{m_1}|\ne\emptyset$ and  $|T_{m_1}|$ is not composed of the same pencil as that of $|C|$, then $\Phi_{|T_{m_1}|}( C)$ is a curve and so $h^0(C, T_{m_1}|_C)\geq 2$. Note that $g( C)=2$.  The Riemann-Roch theorem and the Clifford theorem imply that $(T_{m_1}\cdot C)=\deg(T_{m_1}|_C)\geq 2$, a contradiction.
Hence either $|T_{m_1}|$  is composed of the same pencil as that of $|C|$ on $F$ or $|T_{m_1}|=\emptyset$. Claim \ref{10-}.1 now follows. \end{proof}

\noindent{\bf Claim \ref{10-}.2.}
 {\em Keep the same notation as above. Then $\Phi_{10}$ is birational
under one of the following conditions:
\begin{enumerate}
\item $(T_{m_1}\cdot C)>\frac{m_1}{2}$;

\item $T_{m_1} \cdot C =0$ and $h^0(F, T_{m_1}) > 1+
\frac{m_1}{3}$;

\item $T_{m_1} \geq tC$ for some rational number $t>\frac{m_1}{3}$;

\item either $|T_{m_1}|=\emptyset$ and $P_{m_1}>1+\frac{m_1}{2}$ or $|T_{m_1}|\neq \emptyset$ and $\rounddown{\frac{P_{m_1}-1}{h^0(F,
T_{m_1})}}>\frac{m_1}{2}$.

\item $F_{( r)}$  (resp. $C_{( r)}$) is algebraically equivalent to $F$ (resp. $C$) and $\frac{r+1}{m_1} > \frac{1}{2}$ (resp. $\frac{r+1}{m_1} > \frac{1}{3}$).
\end{enumerate}}

\begin{proof}
 If $(T_{m_1}\cdot C)>\frac{m_1}{2}$, then $\xi\geq \frac{1}{m_1}(S_{m_1}\cdot C)\geq \frac{1}{m_1}(T_{m_1}\cdot C)>\frac{1}{2}$. Then Corollary \ref{cc1} implies that $\Phi_{10}$ is birational, which proves (1).

Now we prove (4). We claim that we have
$$m_1\pi^*(K_X)\geq S_{m_1}\geq rF$$
for an integer $r>\frac{m_1}{2}$. In fact, when $|T_{m_1}|=\emptyset$, $|S_{m_1}|$ is composed of the same pencil as that of $|F|$ and we may take $r:=P_{m_1}-1$. When $|T_{m_1}|\neq \emptyset$, we may take $r=\rounddown{\frac{P_{m_1}-1}{h^0(F, T_{m_1})}}$ and then $S_{m_1}\geq rF$ since $\dim \text{im}(\mu_{m_1})\leq h^0(F, T_{m_1})$.
Then Lemma \ref{cr} implies $\beta \ge
\frac{r}{m_1+r}>\frac{1}{3}$. So $\Phi_{10}$ is birational by
Corollary \ref{cc1}, which asserts (4).

 Since $m_1\pi^*(K_X)|_F\geq T_{m_1} \geq  t C$, we have $\beta> \frac{1}{3}$ and $\Phi_{10}$ is birational by Corollary \ref{cc1}, which proves (3).

If $(T_{m_1}\cdot C)=0$ and $h^0(F, T_{m_1})>1+\frac{m_1}{3}$,
then  $|T_{m_1}|$ is composed of the same pencil as that of $|C|$
and $T_{m_1}\geq tC$ where $t\geq h^0(T_{m_1})-1$. Hence
$\Phi_{10}$ is birational by (3), which proves (2).

Finally, if $F_{(r )}$ is algebraically equivalent to $F$, then
$S_{m_1} \ge F_{(r )}+F \sim (r+1)F$. Hence $\beta \ge
\frac{r+1}{m_1+r+1} > \frac{1}{3}$. Thus $\Phi_{10}$ is birational
by Corollary \ref{cc1}. If $C_{(r )}$ is algebraically equivalent
to $C$, then we have  $\beta \ge \frac{r+1}{m_1} > \frac{1}{3}$ as
well. Hence $\Phi_{10}$ is birational, which verifies (5).
\end{proof}

Return to the proof of Theorem \ref{10-}.

\medskip
\noindent{\bf Part I. $P_3\geq 4$.} Set $m_1=3$.  By Claim
\ref{10-}.2 (1), (2) and Claim \ref{10-}.1, we may assume
$(T_3\cdot C)=0$ and $h^0(F, T_3)\leq 2$. Also by Claim
\ref{10-}.2 (4), we may assume $|T_3|\neq \emptyset$ and $h^0(F,
T_3)=2$.

 By Inequality (\ref{mult1}), one gets $h^0(S_3-F)\geq 2$.
 Clearly we have that $S_3\geq F+F_{(1)}$ and that, by assumption, $F_{(1)}$ is nef. Since $r=1$ and $\frac{r+1}{m_1}=\frac{2}{3}>\frac{1}{2}$, we may assume that $F_{(1)}$ is not algebraically equivalent to $F$ by Claim \ref{10-}.2 (5).

Now clearly we have $h^0(F, F_{(1)}|_F)\geq 2$.  Note that we have
$$|10K_{X'}|\ \lsgeq\  |K_{X'}+\roundup{6\pi^*(K_X)}+F_{(1)}+F|.$$
Kawamata-Viehweg vanishing gives the surjective map:
\begin{eqnarray*}
&&H^0(X', K_{X'}+\roundup{6\pi^*(K_X)}+F_{(1)}+F)\\
&\lrw& H^0(F, K_F+\roundup{6\pi^*(K_X)}|_F+F_{(1)}|_F).
\end{eqnarray*}
It is sufficient to verify the birationality of the
rational map defined by $|K_F+\roundup{6\pi^*(K_X)|_F}+ \Gamma_{(1)}|$ where
 $\Gamma_{(1)} $ is a generic irreducible element in $ \text{Mov} |F_{(1)}|_F|$.

 We claim that $(\pi^*(K_X)\cdot \Gamma_{(1)}) \geq \frac{1}{2}$. In fact, if $\Gamma_{(1)}$ is algebraically equivalent to $C$, then
$(\pi^*(K_X)\cdot \Gamma_{(1)})=\xi\geq \frac{1}{2}$ by Table A3. On the other hand,  if $\Gamma_{(1)}$ is not algebraically equivalent to $C$, then
we should have  $(\Gamma_{(1)}\cdot C)\geq 2$. By Lemma \ref{cr},
$(\pi^*(K_X)|_F\cdot \Gamma_{(1)})\geq \frac{1}{3}(C\cdot \Gamma_{(1)})\geq
\frac{2}{3}$.

Clearly $|K_F+\roundup{6\pi^*(K_X)|_F}+\Gamma_{(1)}|$
distinguishes different generic $\Gamma_{(1)}$'s since
$K_F+\roundup{6\pi^*(K_X)|_F}>0$. Now by the vanishing theorem
again we have the following surjective map:
$$H^0(F, K_F+\roundup{6\pi^*(K_X)|_F}+\Gamma_{(1)})\lrw H^0(\Gamma_{(1)}, K_{\Gamma_{(1)}}+D)$$
where $D:=\roundup{6\pi^*(K_X)|_F}|_{\Gamma_{(1)}}$ with
$\deg(D)\geq 6(\pi^*(K_X)\cdot \Gamma_{(1)})>2$. So $\Phi_{10}$ is
birational by the ordinary birationality principle.
\medskip

\noindent{\bf Part II. $P_4\geq 6$.} We set $m_1=4$.  By Claim \ref{10-}.2 (1) and (4), we may assume
 $(T_4\cdot C)\leq 2$ and $h^0(F, T_4)\geq 2$. Claim \ref{10-}.1 implies either $(T_4\cdot C)=0$ or $(T_4\cdot C)=2$.


\noindent{\bf (II-1)}. If $h^0(F, T_4)=2$, we have $h^0(X',
S_4-2F)\geq 2$ by Inequality (\ref{mult1}). We consider $F_{(2)}$ and  may assume that  $F_{(2)}$ is not algebraically equivalent to $F$ by Claim \ref{10-}.2 (5). Now
$h^0(F, F_{(2)}|_F)\geq 2$ and pick a generic irreducible element
$\Gamma_{(2)}$ of  $\text{Mov}|F_{(2)}|_F|$. By
Kawamata-Viehweg vanishing, we have
\begin{eqnarray*}
|10K_{X'}||_F
&\lsgeq& |K_{X'}+\roundup{5\pi^*(K_X)}+F_{(2)}+2F||_F\\
&=&|K_F+\roundup{5\pi^*(K_X)}|_F+F_{(2)}|_F|\\
&\lsgeq& |K_F+\roundup{5\pi^*(K_X)|_F}+\Gamma_{(2)}|.
\end{eqnarray*}

When $C$ is algebraically equivalent to $\Gamma_{(2)}$(in particular, $C\sim \Gamma_{(2)}$ due to the fact that $q(F)=0$), since
$$\deg({5\pi^*(K_X)}|_{C})=5\xi\geq \frac{5}{2}$$ and
$$|K_F+\roundup{5\pi^*(K_X)|_F}+\Gamma_{(2)}||_C=|K_C+
\roundup{5\pi^*(K_X)|_F}|_{C}|$$ with $\deg(
\roundup{5\pi^*(K_X)|_F}|_{C} )>2$,  we see that $\Phi_{10}|_C$ is
birational by Lemma \ref{R1} and Claim \ref{b2}.1.

When $C$ is not algebraically equivalent to $\Gamma_{(2)}$, we have
$(\Gamma_{(2)}\cdot C)\geq 2$ and
$$K_F+\roundup{5\pi^*(K_X)|_F}+\Gamma_{(2)}\geq
K_F+\roundup{Q_1+C}+\Gamma_{(2)}$$ for certain nef and big
$\bQ$-divisor $Q_1$ on $F$ by Lemma \ref{cr}. The vanishing
theorem also shows that
$$|K_F+\roundup{Q_1}+\Gamma_{(2)}+C||_C=|K_C+(Q_1+\Gamma_{(2)})|_C|$$ gives a birational map since  $\deg((Q_1+\Gamma_{(2)})|_ C)>2$. Thus we have
shown that $\Phi_{10}$ is birational by Lemma \ref{R1} and Claim
\ref{b2}.1.

\noindent{\bf (II-2)}. If $(T_4\cdot C)=0$ and $h^0(F, T_4)\geq
3$, $\Phi_{10}$ is birational by Claim \ref{10-}.2 (2).

\noindent{\bf (II-3)}. If $(T_4\cdot C)=2$ and $h^0(F, T_4)\geq 3$,  then $|T_4|$ is not composed of the same pencil as that of $|C|$ and $h^0(C, T_4|_C)\geq 2$. By the Riemann-Roch and the Clifford theorem, we see $\deg(T_4|_C)=h^0(C,T_4|_C)=2$.  Thus $\dim\text{im}(\nu_4)=2$.

\noindent{\bf (II-3-1)}. If $h^0(F, T_4)\geq 4$, we have $h^0(F,
T_4-C)\geq 2$. Denote by $C_{(1)}$ a generic irreducible element of $\text{Mov}|T_4-C|$. Then we have $T_4\geq C+C_{(1)}$ and we may assume that $C$ is not algebraically equivalent to $C_{(1)}$ by Claim \ref{10-}.2 (5), which implies $(C_{(1)}\cdot C)\geq 2$. By the Kawamata-Viehweg vanishing and
properties of the roundup operator, we have
\begin{eqnarray*}
|10K_{X'}||_F
&\lsgeq&|K_{X'}+\roundup{3\pi^*(K_X)}+S_4+F||_F\\
&=&|K_F+\roundup{3\pi^*(K_X)}|_F+S_4|_F|\\
&\lsgeq&|K_F+\roundup{3\pi^*(K_X)|_F}+C_{(1)}+C|
\end{eqnarray*}
and $$|K_F+\roundup{3\pi^*(K_X)|_F}+C_{(1)}+C||_C=|K_C+D|,$$ where
$D:= (\roundup{3\pi^*(K_X)|_F}+C_{(1)})|_C$ with
$\deg(D)>(C_{(1)}\cdot C)\geq 2$. Thus $\Phi_{10}$ is birational
by  Lemma \ref{R1} and Claim \ref{b2}.1.

\noindent{\bf (II-3-2)}. If $h^0(F, T_4)=3$, we have $h^0(S_4-F)\geq 3$. Again,  we pick a general member $F_{(1)} \in \text{Mov}|S_4-F|$. Consider the natural map:
$$H^0(X', F_{(1)})\overset{\mu_4'}\lrw H^0(F, F_{(1)}|_F)\subset H^0(F, S_4|_F).$$
When $\dim\text{im}(\mu_4')=3$, we see $\dim \nu_4(\text{im}(\mu_4'))=\dim \nu_4(\text{im}(\mu_4))=2$;  when $\dim\text{im}(\mu_4')=2$, we consider the situation $\dim \nu_4(\text{im}(\mu_4'))\leq 1$ at first.  In both cases,  $h^0(F, F_{(1)}|_F-C)>0$ and thus $F_{(1)}|_F-C\geq 0$.  By the vanishing theorem once more, we have
\begin{eqnarray*}
|10K_{X'}||_F&\lsgeq &|K_{X'}+\roundup{5\pi^*(K_X)}+F_{(1)}+F||_F\\
&=& |K_F+\roundup{5\pi^*(K_X)}|_F+F_{(1)}|_F|\\
&\lsgeq&|K_F+\roundup{5\pi^*(K_X)|_F}+C|.
\end{eqnarray*}
Applying the vanishing theorem again, we see
$$|K_F+\roundup{5\pi^*(K_X)|_F}+C||_C=|K_C+D|,$$
where $D:= (\roundup{5\pi^*(K_X)|_F})|_C$ with $\deg(D)\geq
5\xi>2$. Thus $\Phi_{10}$ is birational by Lemma \ref{R1} and
Claim \ref{b2}.1.

 When $\dim\text{im}(\mu_4')=\dim \nu_4(\text{im}(\mu_4'))=2$, then $|F_{(1)}|_F|$ is not composed with the same pencil as that of $|C|$. In particular, $(F_{(1)}\cdot C)\geq 2$. By Lemma \ref{cr}, we have $$K_F+\roundup{5\pi^*(K_X)|_F}+F_{(1)}|_F\geq K_F+\roundup{Q_2+C}+F_{(1)}|_F$$
for certain nef and big $\bQ$-divisor $Q_2$. Since the vanishing
theorem gives $$|K_F+\roundup{Q_2}+F_{(1)}|_F+C||_C=|K_C+D'|$$
with $\deg(D')>(F_{(1)}\cdot C)\geq 2$, we see $\Phi_{10}$ is
birational too by Lemma \ref{R1} and Claim \ref{b2}.1.

Consider the last case $\dim\text{im}(\mu_4')=1$. We see that
$|F_{(1)}|$ is composed of the same pencil as that of $|F|$ and
$F_{(1)}\geq 2F$. Thus $S_4\geq 3F$ and, since
$\frac{3}{m_1}>\frac{1}{2}$, $\Phi_{10}$ is birational by Claim
\ref{10-}.2 (5).
\medskip

\noindent{\bf Part III. $P_5\geq 8$.} We set $m_1=5$. By Claim \ref{10-}.1 and Claim \ref{10-}.2  (1), (2) and (4), we may assume $(T_5\cdot C)=2$ and $h^0(F, T_5)\geq 3$.
Clearly  $|T_5|$ is not composed of the same pencil as that of $|C|$ and so that $h^0(C, T_5|_C)\geq 2$. By the Riemann-Roch and the Clifford theorem, we see $\deg(T_5|_C)=h^0(C,T_5|_C)=2$.  Thus $\dim\text{im}(\nu_5)=2$.

\noindent{\bf (III-1)}. If $h^0(F, T_5)\geq 4$, we have $h^0(F, T_5-C)\geq 2$. Let $C_{(1)}$ be a generic irreducible element in $ \text{Mov}|T_5-C|$.
Thus we have $T_5\geq C+C_{(1)}$ and we may assume that  $C_{(1)}$ is not algebraically equivalent to $C$ by Claim \ref{10-}.2 (5). Hence $(C_{(1)}\cdot C)\geq 2$. By the Kawamata-Viehweg vanishing and properties of the roundup operator, we have the following:
\begin{eqnarray*}
|10K_{X'}||_F
&\lsgeq&|K_{X'}+\roundup{2\pi^*(K_X)}+S_5+F||_F\\
&=&|K_F+\roundup{2\pi^*(K_X)}|_F+S_5|_F|\\
&\lsgeq&|K_F+\roundup{2\pi^*(K_X)|_F}+C_{(1)}+C|
\end{eqnarray*}
and $|K_F+\roundup{2\pi^*(K_X)|_F}+C_{(1)}+C||_C=|K_C+D|$,  with
$$\deg(D)>(C_{(1)}\cdot C)\geq 2.$$ Thus $\Phi_{10}$ is birational by Lemma \ref{R1} and Claim \ref{b2}.1.

\noindent{\bf (III-2)}. If $h^0(F, T_5)=3$, we have $h^0(S_5-F)\geq 5$. Let $F_{(1)} \in \text{Mov}|S_5-F|$ be a general member.  We consider the natural map:
$$H^0(X', F_{(1)})\overset{\mu_5'}\lrw H^0(F, F_{(1)}|_F)\subset H^0(F, S_5|_F).$$
Clearly we have $\dim\text{im}(\mu_5')\leq h^0(F, T_5)=3$.

When $\dim\text{im}(\mu_5')=3$, we see $\dim \nu_5(\text{im}(\mu_5'))=\dim \nu_5(\text{im}(\mu_5))=2$. Thus $|F_{(1)}|_F|$ is not composed of the same pencil as that of $|C|$.  Pick a generic irreducible element $\Gamma_{(1)}$ in the moving part of $|F_{(1)}|_F|$. Then $(\Gamma_{(1)}\cdot C)\geq 2$.
By the vanishing theorem, we have
\begin{eqnarray*}
|10K_{X'}||_F&\lsgeq &|K_{X'}+\roundup{4\pi^*(K_X)}+F_{(1)}+F||_F\\
&=& |K_F+\roundup{4\pi^*(K_X)}|_F+F_{(1)}|_F|\\
&\lsgeq&|K_F+\roundup{4\pi^*(K_X)|_F}+\Gamma_{(1)}|.
\end{eqnarray*}
Applying Lemma \ref{cr}, we have
$$|K_F+\roundup{4\pi^*(K_X)|_F}+\Gamma_{(1)}|\lsgeq |K_F+\roundup{Q_3+C}+\Gamma_{(1)}|$$
where $Q_3$ is certain nef and big $\bQ$-divisor on $F$.  Applying the vanishing once more, we have
$$|K_F+\roundup{Q_3}+\Gamma_{(1)}+C||_C=|K_C+D|$$
with $\deg(D)>(\Gamma_{(1)} \cdot C)\geq 2$.  Thus $\Phi_{10}$ is
birational by Lemma \ref{R1} and Claim \ref{b2}.1.

When $\dim\text{im}(\mu_5')\leq 2$, we have $h^0(X',
F_{(1)}-2F)\geq 1$ and hence $S_5-3F\geq 0$. Therefore $\Phi_{10}$
is birational by Claim \ref{10-}.2 (5).
\medskip

\noindent{\bf Part IV. $P_6\geq 14$.} We set $m_1=6$. By Claim \ref{10-}.1 and Claim  \ref{10-}.2 (1), (2) and (4), we may assume $2\leq (T_6\cdot C)\leq 3$ and $h^0(F, T_6)\geq 4$. Clearly  $|T_6|$ is not composed of the same pencil as that of $|C|$. Thus, by the Riemann-Roch theorem and the Clifford theorem, $\dim\text{im}(\nu_6)=h^0(C, T_6|_C)=2$.

\noindent{\bf  (IV-1)}.  If $h^0(F, T_6)\geq 5$, then we see $h^0(F, T_6-C)\geq 3$.  We pick a general member $C_{(1)}$ in  $\text{Mov}|T_6-C|$.  By Claim \ref{10-}.2 (5), we may assume that  $|C_{(1)}|$ is not composed of the same pencil as that of  $|C|$. We shall analyze the natural map $\nu_6': H^0(F, C_{(1)})\mapsto H^0(C, C_{(1)}|_C)$. Clearly $2\leq \dim\text{im}(\nu_6')\leq h^0(C, T_6|_C)=2$.

 Since $C_{(1)}$ is not algebraically equivalent to $C$, one has  $(C_{(1)}\cdot C)\geq 2$.   By the vanishing theorem, we have
\begin{eqnarray*}
|10K_{X'}||_F&\lsgeq& |K_{X'}+\roundup{\pi^*(K_X)}+S_6+F||_F\\
&\lsgeq&|K_F+\roundup{\pi^*(K_X)|_F}+C_{(1)}+C|
\end{eqnarray*}
and $|K_F+\roundup{\pi^*(K_X)|_F}+C_{(1)}+C||_C=|K_C+D|$ with
$\deg(D)>(C_{(1)}\cdot C)=2$. Thus $\Phi_{10}$ is birational  by
Lemma \ref{R1} and Claim \ref{b2}.1.

\noindent{\bf (IV-2)}. If  $h^0(F, T_6)=4$,  we have $h^0(S_6-F)\geq 10$. We pick a general  member $F_{(1)} \in \text{Mov}|S_6-F|$ and consider the natural map:
$$H^0(X', F_{(1)})\overset{\mu_6'}\lrw H^0(F, F_{(1)}|_F)\subset H^0(F, S_6|_F).$$
Clearly we have $\dim\text{im}(\mu_6')\leq h^0(F, T_6)=4$.

When $\dim\text{im}(\mu_6')\leq 3$, we have $F_{(1)}-3F\geq 0$ and
then $S_6\geq 4F$. By Claim \ref{10-}.2 (5), $\Phi_{10}$ is
birational.

When $\dim\text{im}(\mu_6')=4$, we see $\dim \nu_6(\text{im}(\mu_6'))=\dim \nu_6(\text{im}(\mu_6))=2$. Thus $h^0(F, F_{(1)}|_F-C)=2$.  Furthermore $|F_{(1)}|_F|$ is not composed of the same pencil as that of $|C|$.  Noting that a divisor of degree 1 can not move on $C$, we see $(F_{(1)} \cdot C)\geq 2$.
Denote by $\Gamma_{(1)}$ a general irreducible element of $\text{Mov}|F_{(1)}|_F-C|$.   Noting that $S_6\geq F_{(1)}+F$ and applying the vanishing theorem, we have
\begin{eqnarray*}
|10K_{X'}|&\lsgeq& |K_{X'}+\roundup{3\pi^*(K_X)}+F_{(1)}+F|\\
&\lsgeq& |K_F+\roundup{3\pi^*(K_X)|_F}+F_{(1)}|_F|.
\end{eqnarray*}
If $\Gamma_{(1)}$ is not algebraically equivalent to $C$, we have $(\Gamma_{(1)} \cdot C)\geq 2$.  The vanishing theorem gives
$$|K_F+\roundup{3\pi^*(K_X)|_F}+\Gamma_{(1)}+C||_C=|K_C+D|$$
with $\deg(D)>(\Gamma_{(1)} \cdot C)\geq 2$. Thus $\Phi_{10}$ is
birational by Lemma \ref{R1} and Claim \ref{b2}.1.  If
$\Gamma_{(1)}$ is algebraically equivalent to $C$, we have
$F_{(1)}|_F\geq 2C$ and write
$$F_{(1)}|_F=2C+H_6$$
where $H_6$ is an effective divisor on $F$.  Since $3\pi^*(K_X)|_F+F_{(1)}|_F-C-\frac{1}{2}H_6$ is nef and big, the Kawamata-Viehweg vanishing theorem implies the following surjective map
$$H^0(F, K_F+\roundup{3\pi^*(K_X)|_F+F_{(1)}|_F-\frac{1}{2}H_6})\lrw H^0(C, D')$$
where
$D':=\roundup{3\pi^*(K_X)|_F+F_{(1)}|_F-\frac{1}{2}H_6-C}|_C$ with
$\deg(D')\geq 3\xi+\frac{1}{2}(F_{(1)} \cdot C)>2$.  Thus we see
that $\Phi_{10}$ is birational again by Lemma \ref{R1} and Claim
\ref{b2}.1.  So we conclude the theorem.
\end{proof}

\begin{cor}\label{delt2} (=Theorem \ref{1/2}(2)) Let $X$ be a minimal projective 3-fold of general type with $\delta(X)=2$. If $\Phi_{10}$ is not birational, then the weighted basket ${\mathbb B}(X)=\big(B_X, P_2,\chi(\OO_X)\big)$ are dominated by an initial basket listed in Tables II-1, II-2, II-3 in Appendix.
\end{cor}
\begin{proof}  By Lemma \ref{l1} and Theorem \ref{10-},  we see $0\leq \chi(\OO_X)\leq 3$, $P_2(X)=2$, $P_3(X)\leq 3$, $P_4(X)\leq 5$, $P_5(X)\leq 7$ and $P_6(X)\leq 13$.
According to \cite[Section 3]{Explicit1}, the total number of
numerical types of ${\mathbb B}(X)$ is finite. We  give  a list of
${\mathbb B}^0(X)$ in Tables II-1, II-2 and II-3.
\end{proof}

\section{\bf Projective 4-folds of general type with positive geometric genus}

In order to study 4-folds of general type, we need to prove a slightly general statement on 3-folds.

\begin{thm}\label{pg>0} Let $\nu:\tilde{X}\lrw X$ be a birational morphism from a nonsingular projective
3-fold $\tilde{X}$ of general type onto a minimal model $X$ with
$p_g(X)>0$. Let $Q_\lambda$ be any $\bQ$-divisor on $\tilde{X}$
satisfying $Q_\lambda \equiv \lambda\nu^*(K_X)$ for some rational
number $\lambda >16$. Then $|K_{\tilde{X}}+\roundup{Q_\lambda}|$
gives a birational map onto its image. In particular, $\Phi_m$ is
birational for all $m\geq 18$.
\end{thm}
\begin{proof}    {}From the proof of Corollary \ref{pg=1}, we only need to consider the following two cases:

\noindent {\bf Case 1.} $P_4\geq 2$;\\
\noindent {\bf Case 2.} $P_4=1$ and $P_5\geq 3$.

 Set $m_0=4$ (resp. $5$) in Case $1$ (resp. Case $2$).  Take a sub-pencil $\Lambda\subset |m_0K_X|$.   We use the same set up as in \ref{setup}.
 We may and do assume that $\pi$ factors through $\nu$,  i.e. there is a birational
morphism $\mu:X'\lrw \tilde{X}$ so that $\pi=\nu\circ\mu$ and that
$\mu^*(\{Q_\lambda\})\cup \{\text{exc. divisors of}\ \mu\}$ has
simple normal crossing supports.

Since
$$\mu_*\OO_{X'}(K_{X'}+\roundup{\mu^*(Q_\lambda)})\subseteq
\mu_*\OO_{X'}(K_{X'}+\mu^*\roundup{Q_\lambda})\subseteq
\OO_{\tilde{X}}(K_{\tilde{X}}+\roundup{Q_\lambda}),$$ it is
sufficient to prove the birationality of
$\Phi_{|K_{X'}+\roundup{\mu^*(Q_\lambda)}|}$. We write
$Q_\lambda':=\mu^*(Q_\lambda)\equiv \lambda\pi^*(K_X)$.

We have an induced fibration $f:X'\lrw \Gamma$ onto a smooth
projective curve. Let $F$ be a general fiber of $f$.  Recall that
we have $m_0\pi^*(K_X)\sim_{\bQ} \theta F+E_{\Lambda}'$ for a
positive integer $\theta$ and an effective $\bQ$-divisor $E_{\Lambda}'$ on $X'$.

Without lose of generality, we may assume $p_g(X)=1$ (the case
with $p_g(X)>1$ is much easier). Clearly one has $p_g(F)>0$.
\medskip

\noindent
 {\bf Claim \ref{pg>0}.1.}  One has $h^0(X',
K_{X'}+\roundup{Q_\lambda'})>0$ for  $\lambda>2m_0+1$.\\
By  Lemma \ref{cr},
$$\pi^*(K_X)|_F\equiv \frac{1}{m_0+1}\sigma^*(K_{F_0})+H_{m_0}$$
for certain effective $\bQ$-divisor $H_{m_0}$ on $F$.  Since
$Q_\lambda'-F-\frac{1}{\theta}E_{\Lambda}'\equiv
(\lambda-\frac{m_0}{\theta})\pi^*(K_X)$ is nef and big, Kawamata-Viehweg
vanishing implies the surjective map:
\begin{equation}\label{e1} H^0(X', K_{X'}+\roundup{Q_\lambda'-\frac{1}{\theta}E_{\Lambda}'})\lrw H^0(F, K_F+\roundup{Q_\lambda'-\frac{1}{\theta}E_{\Lambda}'}|_F).\end{equation}
Let
$$\begin{array}{ll} Q_{\lambda,F}:
&=(Q_\lambda'-\frac{1}{\theta}E_{\Lambda}')|_F-(m_0+1)H_{m_0}-\sigma^*(K_{F_0})\\
& \equiv
(\lambda-\frac{m_0}{\theta}-m_0-1)\pi^*(K_X)|_F,\end{array}$$
which is nef and big. Since $p_g(F)>0$, we have
\begin{eqnarray*}
&&h^0(F, K_F+\roundup{Q_\lambda'-\frac{1}{\theta}E_{\Lambda}'}|_F)\\
&\geq& h^0(F, K_F+\roundup{(Q_\lambda'-\frac{1}{\theta}E_{\Lambda}')|_F-(m_0+1)H_{m_0})}\\
&=& h^0(F, K_F+\roundup{Q_{\lambda,F}}+\sigma^*(K_{F_0}))\geq 2
\end{eqnarray*}
by \cite[Lemma 2.14]{Explicit2}. This verifies the Claim.
\medskip

\noindent{\bf Claim \ref{pg>0}.2.} {\em The linear system
$|K_{X'}+\roundup{Q_\lambda'}|$ distinguishes different general
fibers of $f$ for any $\lambda >3m_0+1$. }
\begin{proof}  When $g(\Gamma)=0$, we consider
${Q}_{\zeta}':=Q_\lambda'-F-\frac{1}{\theta}E_{\Lambda}'\equiv
\zeta\pi^*(K_X)$ with
$\zeta=\lambda-\frac{m_0}{\theta}$.
It is clear that $K_{X'}+\roundup{Q_\lambda '}\geq
(K_{X'}+\roundup{{Q}_{\zeta}'})+F$ and hence $|K_{X'}+\roundup{Q_\lambda'}|$
distinguishes different general fibers by Claim \ref{pg>0}.1 since $\zeta>2m_0+1$.

When $g(\Gamma)>0$, we have $\theta \geq 2$. Pick two different
general fibers $F_1$ and $F_2$ of $f$. The vanishing theorem gives
the surjective map:
\begin{eqnarray*}
&&H^0(X', K_{X'}+\roundup{Q_\lambda'-\frac{2}{\theta}E_{\Lambda}'})\\
&\lrw& \oplus_{i=1}^2 H^0(F_i,
(K_{X'}+\roundup{Q_\lambda'-F_1-F_2-\frac{2}{\theta}E_{\Lambda}'}+
F_1+F_2)|_{F_i})
\end{eqnarray*}
where we note that
$(K_{X'}+\roundup{Q_\lambda'-F_1-F_2-\frac{2}{\theta}E_{\Lambda}'})|_{F_i}\geq
0$ due to Claim \ref{pg>0}.1 and the fact $(F_1+F_2)|_{F_i}=0$. Hence $|K_{X'}+\roundup{Q_\lambda'}|$ distinguishes $F_1$ and $F_2$.
\end{proof}

Now we discuss two cases independently.
\medskip

\noindent
{\bf Case 1.} $P_4\geq 2$.\\
If $F$ is a $(1,2)$ surface, we take
$|G|:=\text{Mov}|\sigma^*(K_{F_0})|$ and a general member $C \in |G|$.  By
the surjection map in (\ref{e1}) and Claim \ref{pg>0}.2, it is
sufficient to study the linear system
$|K_F+\roundup{(Q_\lambda'-\frac{1}{\theta}E_{\Lambda}')|_F}|$.
For any $t$, let $$L_{\lambda,
t}:=(Q_\lambda'-\frac{1}{\theta}E_{\Lambda}')|_F-t
\sigma^*(K_{F_0})-5t H_4\equiv
(\lambda-\frac{4}{\theta}-5t)\pi^*(K_X)|_F,$$ which is nef and big
as long as $\lambda-\frac{4}{\theta}-5t >0$. Notice also that
$(Q_\lambda'-\frac{1}{\theta}E_{\Lambda}')|_F \ge L_{\lambda, t} +
t\sigma^*(K_{F_0})$. For simplicity, $L_{\lambda,0}$ is denoted by $L_\lambda$.
In fact, for $\lambda>14$ and by \cite[Lemma 2.14]{Explicit2},  one has
$$
 K_F+\roundup{Q_\lambda'-\frac{1}{\theta}E_{\Lambda}'}|_F
\geq  (K_F+\roundup{L_{\lambda,2}}+\sigma^*(K_{F_0}))+C \geq C.
$$
Thus
$|K_F+\roundup{(Q_\lambda'-\frac{1}{\theta}E_{\Lambda}')|_F}|$
separates different general curves $C$ when $\lambda>14$.
{}Restricting to the curve $C$, one sees by the vanishing theorem
that
$$|K_F+\roundup{(Q_\lambda'-\frac{1}{\theta}E_{\Lambda}')|_F}_{|C}
\ge |K_F+ \roundup{L_{\lambda,1}}+C|_{|C} =
|K_C+\roundup{L_{\lambda,1}}|_C|.$$
Since $\deg(\roundup{L_{\lambda,1}}|_C)\geq (\lambda-\frac{4}{\theta}-5)\xi>2$
for $\xi \ge 2/7$ (cf.  Table A3 with $m_0=4$).  Thus
$\Phi_{|K_{X'}+\roundup{Q_\lambda'}|}$ separates points on the general curve $C$ and
hence is birational when $\lambda>16$.

Assume that $F$ is not a (1,2) surface. We would like to study
$|K_F+\roundup{L_\lambda}|$ where
$L_\lambda:=(Q_\lambda'-\frac{1}{\theta}E_{\Lambda}')|_F$, making use of the
relation (\ref{e1}).  If $K_{F_0}^2\geq 2$, Inequalities (\ref{j3}), (\ref{j5})
imply
$$L_\lambda^2\geq \frac{2(\lambda-4)^2}{25}>8$$
whenever $\lambda> 14$.  If $F$ is a $(1,1)$ surface, then we have $q(X)=g(\Gamma)\geq 0$ and $h^2(\OO_X)=0$ as seen in the proof of Case 2 of Corollary \ref{pg=1}.
Hence we have $\chi(\OO_X)\leq 0$ and Reid's Riemann-Roch formula gives $P_5>P_4\geq 2$. In particular, we have $P_5\geq 3$.  We omit the discussion for the situation when $|5K_{X'}|$ and $|4K_{X'}|$ are composed with the same pencil since that is a comparatively much better case. So may assume that $|5K_{X'}||_F$ is moving on $F$. If we take $|G_1|:=\text{Mov}|\roundup{5\pi^*(K_X)}|_F|$, we have $\beta_{G_1}=\frac{1}{5}$. Then,
by Lemma \ref{cr} and Lemma \ref{>1}, we have
$$L_\lambda^2\geq \frac{(\lambda-4)^2}{25}(\sigma^*(K_{F_0})\cdot G_1)\geq
\frac{2(\lambda-4)^2}{25}>8$$
whenever $\lambda> 14$.
{}Finally, for both cases,  $(L_\lambda\cdot
\tilde{C})\geq \frac{2(\lambda-4)}{5}\geq 4$ for $\lambda\geq 14$
and for any very general curve $\tilde{C}$ on $F$. Therefore, by
Lemma \ref{masek}, $|K_{F}+\roundup{L_\lambda}|$ gives a birational map
when $\lambda \geq 14$.

Hence, when $P_4\geq 2$, $\Phi_{|K_{X'}+\roundup{Q'_\lambda}|}$ is
birational for $\lambda >16$.
\medskip

\noindent
{\bf Case 2.}  $P_4=1$ and $P_5\geq 3$.\\
We set $m_0=5$.
If $d_5=1$, we set $\Lambda=|5K_X|$. Then we are in much better situation than that of $P_3=2$  since we have $\theta \geq 2$ (and noting that $\frac{\theta}{m_0}=\frac{2}{5}>\frac{1}{3}$).
We omit the details and leave this as an exercise to interested
readers.

If $d_5\geq 2$, we take a sub-pencil $\Lambda\subset|5K_{X}|$ and $\Lambda$ induces a fibration $f:X'\lrw \Gamma$ onto a smooth complete curve $\Gamma$. As we have seen in Case 3 of Corollary \ref{pg=1}, the general fiber $F$ satisfies $K_{F_0}^2\geq 2$. For the similar reason, we can take $m_1=5$ and $|G|:=\text{Mov}|m_1K_{X'}|_F|$. 
Pick a generic irreducible element $C$ in $|G|$. Lemma \ref{cr}
implies $\xi=(\pi^*(K_X)\cdot C)\geq
\frac{1}{6}(\sigma^*(K_{F_0})\cdot C)\geq \frac{1}{3}$. We may
write $5\pi^*(K_X)|_F\equiv C+N_5$ for an effective $\bQ$-divisor
$N_5$ on $F$. For two different generic irreducible curves $C_1$ and $C_2$ in
$|G|$, we set
$$L_{\lambda,2}:=(Q_\lambda'-\frac{1}{\theta}E_{\Lambda}')|_F-C_1-C_2-2N_5,$$
and
$$L_{\lambda,1}:=(Q_\lambda'-\frac{1}{\theta}E_{\Lambda}')|_F-C-N_5$$
respectively. It is clear that they are both nef and big
 whenever $\lambda>15$.

Thanks to the  vanishing theorem, we have the surjective map:
\begin{eqnarray*}
H^0(F, K_F+\roundup{L_{\lambda}-2N_5})&\lrw&
H^0(C_1, K_{C_1}+\roundup{L_{\lambda,2}}|_{C_1}+C_2|_{C_1})\\
&&\oplus\  H^0(C_2, K_{C_2}+\roundup{L_{\lambda,2}}|_{C_2}+C_1|_{C_2})
\end{eqnarray*}
if $\lambda >15$. It is
clear that
$$H^0(C_i,K_{C_i}+ \roundup{L_{\lambda,2}}_{|{C_i}}+C_{2-i}|_{C_i}) \ne
0$$
since $L_{\lambda,2}$ is nef and big.  Hence $|K_F+\roundup{(Q_\lambda'-\frac{1}{\theta}E_{\Lambda}')|_F-2N_5}|=|K_F+\roundup{L_{\lambda}-2N_5}|$
separates different general curves $C$ in $|G|$.
This also implies that
$|K_F+\roundup{(Q_\lambda'-\frac{1}{\theta}E_{\Lambda}')}|$ can
distinguishes $C_1$ and $C_2$. 
Now applying the vanishing theorem once more, we get the
surjective map: $$H^0(F, K_F+\roundup{L_{\lambda}-N_5}) \lrw
H^0(C, K_{C}+\roundup{L_{\lambda,1}}_{|{C}})$$
with
$$\deg(\roundup{L_{\lambda,1}}_{|{C}})\geq (\lambda-\frac{5}{\theta}-5)\xi>2$$
whenever $\lambda>16$ for $\xi \ge 1/3$. Thus, by Theorem
\ref{birat}, $|K_{X'}+\roundup{Q_\lambda'}|$ gives a birational
map for $\lambda>16$. So we conclude the statement of the theorem.
\end{proof}

\begin{thm}\label{b4}(=Theorem \ref{4folds}) Let $V$ be a nonsingular projective 4-fold of general type. Then,
\begin{itemize}
\item[(1)] when $p_g(V)\geq 2$, $\Phi_{m,V}$ is birational for all
$m\geq 35$; \item[(2)] when $p_g(V)\geq 19$, $\Phi_{m,V}$ is
birational for all $m\geq 18$.
\end{itemize}
\end{thm}
\begin{proof}
Let $Z$ be the minimal model of $V$.
We set $m_0=1$, $\Lambda=|K_Z|$ and use the set up in \ref{setup}.  Thus we have an induced fibration $f:Z'\lrw \Gamma$.

{}First we consider the case $\dim \Gamma=1$. Recall that we have
$M_{\Lambda}\equiv  \theta F$ for a general fiber $F$ of $f$,
where $\theta \geq p_g(Z)-1$.  It is clear that, when $m\geq 3$ ,
$|mK_{Z'}|$ distinguishes different general fibers of $f$. Pick a
general fiber $F=X'$, which is a nonsingular projective 3-fold of
general type with $p_g(X')>0$. Replace by its birational model, we
may assume that there is a birational morphism $\nu:X'\lrw X$ onto
a minimal model.   By Lemma \ref{cr}, we have
$$\pi^*(K_Z)|_{X'}\equiv\frac{\theta}{\theta+1}\nu^*(K_{X})+J_1$$
for an effective $\bQ$-divisor $J_1$ on $X'$. When $m$ is large,
since $(m-1)\pi^*(K_Z)-X'-\frac{1}{\theta}E_{\Lambda}'$ is nef and
big, Kawamata-Viehweg vanishing implies:
\begin{eqnarray*}
&&|K_{Z'}+\roundup{(m-1)\pi^*(K_Z)-\frac{1}{\theta}E_{\Lambda}'}||_{X'}\\
&=&|K_{X'}+\roundup{(m-1)\pi^*(K_Z)-\frac{1}{\theta}E_{\Lambda}'}_{X'}|\\
&\lsgeq&|K_{X'}+\roundup{R_m}|
\end{eqnarray*}
where
$R_m:=((m-1)\pi^*(K_Z)-X'-\frac{1}{\theta}E_{\Lambda}')|_{X'}$.
In fact, we have
\begin{eqnarray*}
R_m&\equiv& (m-1-\frac{1}{\theta})
\pi^*(K_Z)|_{X'}\\
&\equiv &
(\frac{m\theta}{\theta+1}-1)\nu^*(K_X)+(m-1-\frac{1}{\theta})J_1
\end{eqnarray*}
Since $\frac{m\theta}{\theta+1}-1>16$ whenever either $m\geq
18$ and $p_g(Z)\geq 19$ or $m\geq 35$ and $p_g(Z)\geq 2$, Theorem
\ref{pg>0} implies that
$|K_{X'}+\roundup{R_m-(m-1-\frac{1}{p})J_1}|$ gives a birational
map. Thus statements of the theorem follow in this case.

Next we consider the case $\dim \Gamma\geq 2$. By definition, $\theta=1$. Clearly it is sufficient to consider $\Phi_{|mK_{Z'}|}|_{X'}$ for a general member $X'\in |M_{\Lambda}|$. We consider a general $X'$ and, similarly, we may assume that there is a birational morphism $\nu:X'\lrw X$ onto a minimal model $X$. Then Kawamata's extension theorem \cite[Theorem A]{EXT} still implies
\begin{equation}\pi^*(K_Z)|_{X'}\geq \frac{1}{2}\nu^*(K_X).\label{a1}\end{equation}
We consider the linear system $|M_{\Lambda}|_{X'}|$, which may be assumed to be base point free modulo further birational modifications. Pick a generic irreducible element $S$ of this linear system. We clearly have
$$\pi^*(K_Z)|_{X'}\geq M_{\Lambda}|_{X'}\geq S.$$
Modulo $\bQ$-linear equivalence, one has
$$2S\leq (\pi^*(K_Z)+X')|_{X'}\leq K_{X'}.$$
Thus Kawamata's extension theorem gives
\begin{equation} \nu^*(K_{X})|_S\geq \frac{2}{3}\sigma^*(K_{S_0})\label{a2}\end{equation}
where $\sigma:S\lrw S_0$ is the contraction onto the minimal model $S_0$ of $S$.
Both (\ref{a1}) and (\ref{a2}) imply
$$\pi^*(K_Z)|_{S}\geq \frac{1}{3}\sigma^*(K_{S_0}).$$
Write $\pi^*(K_Z)|_{X'}\equiv S+H_{\Lambda}$ where $H_{\Lambda}$ is an effective $\bQ$-divisor on $X'$.
Since $R_m-S-H_{\Lambda}\equiv (m-3)\pi^*(K_Z)|_{X'}$ is nef and big, the vanishing theorem implies
$$\begin{array}{ll}
|K_{X'}+\roundup{R_m-H_{\Lambda}}|_{|S}
&=|K_S+\roundup{R_m-S-H_{\Lambda}}_{|S}|\\
&\lsgeq |K_S+\roundup{R_{m,S}}|
\end{array}$$
where $R_{m,S}:=(R_m-S-H_{\Lambda})|_S$. Note that
$$\begin{array}{ll}
R_{m,S}&\equiv (m-3)\pi^*(K_Z)|_S\\
&\equiv \frac{m-3}{3}\sigma^*(K_{S_0})+E_{m,S}
\end{array}$$
where $E_{m,S}$ is an effective $\bQ$-divisor on $S$. Now it is
clear by Lemma \ref{masek} that $|K_S+\roundup{R_{m,S}-E_{m,S}}|$
gives a birational map whenever $m\geq 15$. Again Kawamata-Viehweg
vanishing shows that $|K_{X'}+\roundup{R_m}|$ distinguishes
different elements $S$. Thus we have shown that $\Phi_{m,Z}$ is
birational for all $m\geq 15$ in this case. We are done.
\end{proof}

G. Brown and M. Reid kindly informed us the following interesting
canonical 4-folds:

\begin{exmp} The general hypersurfaces $W_{36}\subset \bP(1,1,3,5,7,18)$ and $Y_{36}\subset\bP(1,1,4,5,6,18)$ have canonical singularities, $p_g=2$. It is clear that the 17-canonical maps of these two 4-folds are not birational.
\end{exmp}

\begin{prbm} It is a very interesting problem to find more examples of 4-folds of general type so that $\Phi_m$ is not birational for large $m$.
\end{prbm}

\section{\bf Appendix. Tables}

\centerline{\textbf{Table F-0}}
\begin{center}
\noindent{\tiny
\begin{tabular}{|l|l|l|l|l|}
\hline
 Types & $B_X$ & $\chi$& $K_X^3$& $\delta(X)$\\
\hline
2a& $\{4\times(1, 2), (4, 9),  (2, 5), (5, 13), 3\times(1,3), 2\times(1, 4)\}$&    2   &1/1170 &18\\
\hline
41& $\{5\times(1, 2), (4, 9), 2\times(3, 8), (1, 3), 2\times(2,7)\}$&2&1/252&  13\\
\hline
\end{tabular}}
\end{center}
\smallskip

\centerline{\textbf{Table F-1}}
\begin{center}
\noindent{\tiny
\begin{tabular}{|l|l|l|l|l|}
\hline
 Types & $B_X$ & $\chi$& $K_X^3$& $\delta(X)$\\
\hline
2& $\{4\times(1, 2), (4, 9), 2\times(2, 5), (3, 8), 3\times(1, 3), 2\times(1, 4)\}$&    2   &1/360& 15\\
\hline
3&$\{6\times(1, 2), (5, 11), 4\times(2, 5), (3, 8), 4\times(1, 3), (2, 7), 2\times(1, 4)\}$&    3   &23/9240    &15\\
\hline
5.1& $\{7\times(1, 2), (4, 9),3\times(2, 5), (5, 13), 4\times(1, 3), (3, 11), (1, 4)\}$&   3   &61/25740   &15\\
\hline
5.2&$\{7\times(1,2), (4, 9), 2\times(2,5), (7, 18), 4\times(1,3), (3, 11), (1, 4)\}$&  3   &1/660  &15\\
\hline
5.3&$\{7\times(1,2), (4, 9), (2, 5), (9, 23), 4\times(1,3), (3, 11), (1, 4)\}$&   3   &47/45540   &15\\
\hline
5a& $\{7\times(1,2), (4, 9), (11, 28), 4\times(1,3), (3, 11), (1, 4)\}$&3 &1/1386&    15\\
\hline
5b& $\{7\times(1,2), (4, 9), 3\times(2,5), (5, 13), 4\times(1,3),(4, 15)\}$ &  3   &1/1170 &15\\
\hline
53a &$\{3\times(1,2), (4, 9), 2\times(2,5), (5, 13), 3\times(1,3),  (1, 5)\}$& 2   &1/1170 &15\\
\hline
\end{tabular}}
\end{center}
\smallskip

\centerline{\textbf{Table F-2}}
\begin{center}
\noindent{\tiny
\begin{tabular}{|l|l|l|l|l|}
\hline
 Types & $B_X$ & $\chi$& $K_X^3$& $\delta(X)$\\
\hline
1&  $\{5\times(1, 2),  (3, 7), 3\times (2, 5), 3\times(1, 3), (3, 11)\}$&  2   &3/770  &14\\
\hline
4& $\{7\times(1,2),(4, 9), 4\times(2, 5), (4, 11), 3\times(1, 3), (2, 7), 2\times(1, 4)\}$& 3&  13/3465 &14\\
\hline
4.5 &$\{7\times(1, 2), (4, 9), 4\times(2, 5), (5, 14), 2\times(1, 3),  (2, 7), 2\times(1, 4)\}$ &3& 1/630&  14\\
\hline
5&  $\{7\times(1, 2), (4, 9), 4\times(2, 5), (3, 8), 4\times(1, 3),  (3, 11), (1, 4)\}$&   3   &17/3960&   14\\
\hline
5.4 &$\{7\times(1, 2), (4, 9), 4\times(2, 5), (3, 8), 4\times(1,3), (4, 15)\}$&    3   &1/360& 14\\
\hline
6&  $\{9\times(1,2), 2\times(3,7), (2, 5), (4, 11), 4\times(1,3), 2\times(2,7), (1, 5)\}$&  3   &1/462  &14\\
\hline
7&  $\{5\times(1,2), (4, 9), (3, 7), 5\times(1,3), (2, 7), (1, 5)\}$& 2   &1/630  &14\\
\hline
7a& $\{5\times(1,2),(7, 16), 5\times(1,3), (2, 7), (1, 5)\}$& 2&  1/1680  &14\\
\hline
10& $\{8\times(1,2), (4, 9), (3, 7), 2\times(3,8),5\times(1,3), (2, 7), (1, 4), (1, 5)\}$& 3   &1/630& 14\\
\hline
11&$\{  9\times(1,2), 2\times(3,7),(3, 8), (4, 11), 3\times(1,3),(3, 10), (1, 4), (1, 5)\}$&   3   &3/3080&    14\\
\hline
12& $\{9\times(1,2), (4, 9), (2, 5), 2\times(3,8),4\times(1,3), 2\times(2,7),  (1, 5)\}$&   3   &1/252  &14\\
\hline
12.1&$\{9\times(1,2),(4, 9), (5, 13), (3, 8), 4\times(1,3), 2\times(2,7), (1, 5)\}$&   3&  67/32760&   14\\
\hline
12a&$\{9\times(1,2), (4, 9), (8, 21), 4\times(1,3), 2\times(2,7), (1, 5)\}$&3  &1/630& 14\\
\hline
14&$\{  10\times(1,2), (3, 7), 2\times(2,5), 2\times(3,8), 6\times(1,3), 2\times(2,7), $&&&\\
&$(1, 4), (1, 5)\}$&4   &1/770  &14\\
\hline
15&$\{11\times(1,2), (4, 9), (3, 7), 2\times(2,5), (3, 8), (4, 11), 5\times(1,3), 2\times(2,7), $&&&\\
&$(1, 4), (1, 5)\}$&  4&  71/27720&   14\\
\hline
15.1&$\{11\times(1,2), (4, 9), (3, 7), 2\times(2,5), (7, 19), 5\times(1,3),2\times(2,7), $&&&\\
&$(1, 4),(1, 5)\}$&  4   &47/23940&  14\\
\hline
15.2&$\{11\times(1,2), (7, 16), 2\times(2,5), (3, 8), (4, 11), 5\times(1,3), $&&&\\
&$2\times(2,7),    (1, 4), (1, 5)\}$&   4&  29/18480&   14\\
\hline
16&$\{11\times(1,2),(4, 9), (3, 7), 2\times(2,5), 2\times(3,8), 6\times(1,3), (2, 7), $&&&\\
&$(3, 11), (1, 5)\}$&   4&  43/13860&   14\\
\hline
\end{tabular}}\end{center}

\begin{center}
\noindent{\tiny
\begin{tabular}{|l|l|l|l|l|}

\hline
16.1&$\{11\times(1,2), (4, 9), (3, 7), (2, 5), (5, 13), (3, 8), 6\times(1,3),(2, 7), $&&&\\
&$(3, 11), (1, 5)\}$&  4   &85/72072&  14\\
\hline
16.2&$\{11\times(1,2), (7, 16), 2\times(2,5), 2\times(3,8), 6\times(1,3), (2, 7), $&&&\\
&$(3, 11), (1, 5)\}$&   4   &13/6160&   14\\
\hline
16.4&$\{11\times(1,2), (7, 16), 2\times(2,5), 2\times(3,8), 6\times(1,3), (5, 18), (1, 5)\}$&   4&  1/720   &14\\
\hline
16.5&$\{11\times(1,2), (4, 9), (3, 7), 2\times(2,5), 2\times(3,8), 6\times(1,3), (5, 18),$&&&\\
&$(1, 5)\}$& 4   &1/420  &14\\
\hline

17  &$\{9\times(1,2), 2\times(3,7), 2\times(4,11), 3\times(1,3), (2, 7), (1, 4), (1, 5)\}$& 3&  3/1540  &14\\
\hline

18& $\{9\times(1,2), 2\times(3,7), (3, 8), (4, 11), 4\times(1,3), (3, 11), (1, 5)\}$   &3  &23/9240&   14\\
\hline
18b&    $\{9\times(1,2), 2\times(3,7), (7,19), 4\times(1,3), (3, 11), (1, 5)\}$    &3& 83/43890&   14\\
\hline

20  &$\{7\times(1,2), 2\times(4,9), (2, 5), (3, 8), 6\times(1,3), (2, 7), (1, 4), (1, 5)\}$&   3&  1/504   &14\\
\hline
21  &$\{6\times(1,2), (4, 9), (3, 8), 3\times(1,3), (3, 10), (1, 5)\}$&   2   &1/360& 14\\
\hline
23  &$\{8\times(1,2), (4, 9), (3, 7), (2, 5), (4, 11), 4\times(1,3), (3, 10), (1, 4), $&&&\\
&$(1, 5)\}$&  3   &19/13860   &14\\
\hline

25  &$\{9\times(1,2), (5, 11), (4, 9), 3\times(2,5), (3, 8), 7\times(1,3), 2\times(2,7), $&&&\\
&$(1, 4), (1, 5)\}$& 4   &47/27720&  14\\
\hline
25a &$\{9\times(1,2), (9, 20), 3\times(2,5), (3, 8), 7\times(1,3), 2\times(2,7), (1, 4), $&&&\\
&$(1, 5)\}$& 4   &1/840  &14\\
\hline
26  &$\{10\times(1,2), 2\times(4,9), 3\times(2,5), (4, 11), 6\times(1,3), 2\times(2,7), $&&&\\
&$(1, 4), (1, 5)\}$&   4   &41/13860&  14\\
\hline
27& $\{10\times(1,2), 2\times(4,9), 3\times(2,5), (3, 8), 7\times(1,3), (2, 7), $&&&\\
&$(3, 11), (1,5)\}$   &4  &97/27720&  14\\
\hline
27.3&$\{10\times(1,2), 2\times(4,9), 3\times(2,5), (3,8), 7\times(1,3), (5, 18), (1, 5)\}$  &4  &1/360& 14\\
\hline
28&$\{5\times(1,2), (5, 11), (3, 8), 4\times(1,3), (2, 7), (1, 5)\}$  &2  &23/9240&   14\\
\hline 29& $\{6\times(1,2), (4,9), (4, 11), 3\times(1,3), (2, 7), (1, 5)\}$& 2&  13/3465 &14\\
\hline
29.1&$\{6\times(1,2), (4, 9), (5, 14), 2\times(1,3), (2,7), (1, 5)\}$ &2& 1/630&  14\\
\hline
30  &$\{7\times(1,2), (5,11), (3,7), (2,5), (4, 11), 5\times(1,3), (2,7), (1,4), (1,5)\}$&    3&  1/924&  14\\
\hline
31  &$\{7\times(1,2), (5,11), (3,7), (2,5), (3,8), 6\times(1,3), (3,11), (1,5)\}$&3&  1/616&  14\\
\hline
32  &$\{8\times(1,2), (4,9), (3,7), (2,5), (4, 11), 5\times(1,3), (3,11), (1,5)\}$&3  &2/693  &14\\
\hline

32a &$\{8\times(1,2), (7,16), (2, 5), (4, 11), 5\times(1,3), (3,11), (1,5)\}$&3&  1/528&  14\\
\hline
33& $5\times(1,2), 2\times(3,7), (3,8), (1,3), (3,10), (2,7)\}$&  2   &1/840& 14\\
\hline
34& $\{7\times(1,2), (4,9), (3,7), 2\times(2,5), (3,8), 3\times(1,3), 3\times(2,7)\}$&  3&  1/360&  14\\
\hline
34a &$\{7\times(1,2), (7, 16), 2\times(2,5), (3, 8), 3\times(1,3), 3\times(2,7)\}$  &3& 1/560   &14\\
\hline
35  &$\{5\times(1,2), 2\times(3,7), (4,11), (1, 3), 2\times(2,7)\}$&   2&  1/462&  14\\
\hline
36  &$\{4\times(1,2), (4,9), (3,7), (2, 5), 2\times(1,3), (3,10), (2,7)\}$&   2&  1/630   &14\\
\hline
36a &$\{4\times(1,2), (7,16), (2, 5), 2\times(1,3), (3,10), (2,7)\}$& 2   &1/1680&    14\\
\hline
36b &$\{4\times(1,2), (4, 9), (3,7), (2,5), 2\times(1,3), (5, 17)\}$  &2& 4/5355& 14\\
\hline
37  &$6\times(1,2), 2\times(4,9), 3\times(2,5), 4\times(1,3), 3\times(2,7)\}$&   3   &1/315  &14\\
\hline
38  &$\{3\times(1,2), (5,11), (3,7), (2, 5), 3\times(1,3), 2\times(2,7)\}$&    2   &1/770& 14\\
\hline
39  &$\{7\times(1,2), (4,9), (3,7),(2, 5), 2\times(3,8), 2\times(1,3), (3,10), (2,7), (1,4)\}$&3&  1/630   &14\\
\hline
40& $\{9\times(1,2), 2\times(4, 9), 3\times(2,5), 2\times(3,8), 4\times(1,3), 3\times(2,7), (1,4)\}$& 4   &1/315& 14\\
\hline
42  &$\{6\times(1,2), (5,11), (3, 7), (2, 5), 2\times(3,8), 3\times(1,3), 2\times(2,7), (1,4)\}$&   3&  1/770&  14\\
\hline
43  &$\{7\times(1,2), (4,9), (3,7), (2, 5), (3, 8), (4, 11), 2\times(1,3), 2\times(2,7), (1, 4)\}$ &3& 71/27720&   14\\
\hline
43.1&   $\{7\times(1,2), (7, 16), (2, 5), (3,8), (4,11), 2\times(1,3), 2\times(2,7), (1,4)\}$& 3   &29/18480&  14\\
\hline
43c &$\{7\times(1,2), (7,16), (2, 5), (7, 19), 2\times(1,3), 2\times(2,7), (1, 4)\}$&  3   &31/31920&  14\\
\hline
43.2&   $\{7\times(1,2), (4,9), (3, 7), (2, 5), (7, 19), 2\times(1,3), 2\times(2,7),(1, 4)\}$  &3  &47/23940&  14\\
\hline
44  &$\{7\times(1,2), (4,9), (3, 7), (2, 5), 2\times(3,8), 3\times(1,3), (2, 7), (3, 11)\}$ &  3   &43/13860   &14\\
\hline
44.1&   $\{7\times(1,2), (4, 9), (3, 7), (5, 13), (3, 8), 3\times(1,3), (2, 7), (3, 11)\}$&   3   &85/72072&  14\\
\hline
44.2&   $\{7\times(1,2), (4, 9), (3, 7), (2, 5), 2\times(3,8), 3\times(1,3), (5, 18)\}$ &  3   &1/420  &14\\
\hline
44.3    &$\{7\times(1,2), (7, 16), (2, 5), 2\times(3,8), 3\times(1,3), (2,7), (3, 11)\}$&3 &13/6160&   14\\
\hline
44c &$\{7\times(1,2), (7, 16), (2, 5), 2\times(3, 8), 3\times(1,3), (5, 18)\}$ &3& 1/720&  14\\
\hline
45  &$\{3\times(1,2), 2\times(4, 9), (3, 8), 3\times(1,3), (2, 7), (1, 4)\}$   &2& 1/504   &14\\
\hline
46& $\{6\times(1,2), 2\times(4,9), 2\times(2,5), (3, 8), 3\times(1,3), (3, 10), (2, 7), (1, 4)\}$&  3   &1/504  &14\\
\hline
46b &$\{6\times(1,2), 2\times(4,9), 2\times(2,5), (3, 8), 3\times(1,3), (5, 17), (1, 4)\}$&3    &7/6120&    14\\
\hline
48  &$\{4\times(1,2), (4,9), (3, 7), (4, 11), (1, 3), (3, 10), (1, 4)\}$&    2&  19/13860    &14\\
\hline
49& $\{5\times(1,2), (5, 11), (4, 9), 2\times(2,5), (3, 8), 4\times(1,3), 2\times(2,7), (1,4)\}$&3  &47/27720&  14\\
\hline
49a &$\{(5\times(1,2),  (9, 20), 2\times(2,5), (3, 8), 4\times(1,3), 2\times(2,7), (1, 4)\}$&   3&  1/840   &14\\
\hline
50  &$\{6\times(1,2), 2\times(2,9), 2\times(2,5), (4, 11), 3\times(1,3), 2\times(2,7), (1, 4)\}$ &3  &41/13860&  14\\
\hline
51& $\{6\times(1,2), 2\times(4,9), 2\times(2,5), (3, 8), 4\times(1,3), (2, 7), (3, 11)\}$&  3   &97/27720&  14\\
\hline
51.1&$\{6\times(1,2),2\times(4,9), (2, 5), (5, 13), 4\times(1,3), (2, 7), (3, 11)\}$&  3   &71/45045&  14\\
\hline
52  &$\{4\times(1,2), (3, 7), 2\times(2,5), 2\times(3,8),  2\times(1,3), (1, 5)\}$&2&   1/420&  14\\
\hline

53& $3\times(1,2),(4, 9), 3\times(2,5), (3,8), 3\times(1,3), (1, 5)\}$&    2   &1/360& 14\\
\hline
54  &$\{2\times(1,2), 2\times(3,7), 3\times(2,5), (3, 8), (1, 3), (2, 7)\}$    &2& 1/840&  14\\
\hline
\end{tabular}}\end{center}

\begin{center}
\noindent{\tiny
\begin{tabular}{|l|l|l|l|l|}
\hline

56  &$\{(1, 2), (4, 9), (3, 7), 4\times(2,5), 2\times(1,3), (2, 7)\}$ &2& 1/630&  14\\
\hline
58  &$\{4\times(1,2), (4,9), (3, 7), 4\times(2,5), 2\times(3,8), 2\times(1,3), (2, 7), (1, 4)\}$&   3   &1/630  &14\\
\hline
59  &$\{2\times(1,2), 2\times(3,7), 2\times(2,5), (3, 8), (4, 11), (1,4)\}$&   2&  3/3080  &14\\
\hline
60  &$3\times(1,2), 2\times(4,9), 5(2,5), (3,8), 3\times(1,3), (2,7), (1,4)\}$ &3  &1/504& 14\\
\hline
62  &$\{(1, 2), (4, 9), (3, 7), 3\times(2,5), (4, 11), (1, 3), (1, 4)\}$&    2&  19/13860&   14\\
\hline
\end{tabular}}
\end{center}


\vbox{\centerline{\textbf{Table II--1}}

\begin{center}
\noindent{\tiny
\begin{tabular}{|l|l|l|l|l|}
\hline
No. & $B^0(X)$ &  $K_X^3$&$\chi$ &$(P_3,P_4,P_5,P_6)$\\
\hline
1&  $\{5*(1, 2), 2 * (1, 3)\}$&{1}/{6}&0 & $(3,5,7,11)$\\
\hline
2&  $\{5*(1, 2), (1, 3), (1,4)\}$&{1}/{12}&0 & $(3,5,6,9)$\\
\hline
3&$\{18 *(1, 2),  (1, 3), \}$&{1}/{3}&1 & $(1,5,6,13)$\\
\hline

4&$\{(18-4t) *(1, 2), 3t * (1,3),  (1, 4) \}, t=0,1,2$&{1}/{4}&1 & $(1+t,5,5+t,11+t)$\\
\hline

5&$\{(18-4t) *(1, 2), 3t * (1,3),  (1, 5) \},5 \le r \le 12; t=0,1,2$&{1}/{r}&1 & $(1+t,5,5+t,10+t)$\\
\hline
6&$\{(17-4t) *(1, 2), (2+3t) * (1, 3)\}, t=0,1,2$&{1}/{6}&1 & $(1+t,4,4+t,9+t)$\\
\hline
7&$\{(14-4t) *(1, 2), (2+3t) * (1, 3), 2 * (1,4)\}, t=0,1$&{1}/{6}&1 & $(2+t,5,5+t,10+t)$\\
\hline
8&$\{(14-4t) *(1, 2), (2+3t) * (1, 3),  (1,4), (1,5)\}, t=0,1$&{7}/{60}&1 & $(2+t,5,5+t,9+t)$\\
\hline
9&$\{(14-4t) *(1, 2), (2+3t) * (1, 3),  (1,4), (1,6)\}, t=0,1$&{1}/{12}&1 & $(2+t,5,5+t,9+t)$\\
\hline
10&$\{(14-4t) *(1, 2), (1+3t) * (1, 3), 3 * (1,4)\}, t=0,1$&{1}/{12}&1 & $(2+t,5,4+t,8+t)$\\
\hline
11&$\{(17-4t) *(1, 2), (1+3t) * (1, 3), (1,4)\}, t=0,1,2$&{1}/{12}&1 & $(1+t,4,3+t,7+t)$\\
\hline
\end{tabular}}\end{center}}
\vskip0.5cm

\vbox{\centerline{\textbf{Table II-2}}
\medskip

\begin{center}
\noindent{\tiny
\begin{tabular}{|l|l|l|l|l|}
\hline
No. & $B^0(X)$ &  $K_X^3$&$\chi$ &$(P_3,P_4,P_5,P_6)$\\
\hline
1&$\{27 *(1, 2), 2 * (1,3), (1,r)\}$&$\frac{1}{6}+\frac{1}{r}$&2 & $(0,5,5,13)$\\
\hline
2&$\{(27-4t) *(1, 2), (1+3t) * (1,3)$, & & &\\
&$2 * (1,4)\}, t=0,1.$&{1}/{3}&2 & $(t,5,4+t,12+t)$\\
\hline
3&$\{(27-4t) *(1, 2), (1+3t) * (1,3), $&&&\\
&$(1,4), (1,r)\},5 \le r; t=0,1,2.$&$\frac{1}{12}+\frac{1}{r}$&2 & $(t,5,4+t,11+t)$\\
\hline
4&$\{(27-4t) *(1, 2), (1+3t) * (1,3), $&&&\\
&$(1,r_1), (1,r_2) \}, (r_1,r_2) \in I_4; t=0,1,2,3.$&$\frac{1}{r_1}+\frac{1}{r_2}-\frac{1}{6}$&2 & $(t,5,4+t,10+t)$\\
\hline
5&$\{(26-4t) *(1, 2), (4+3t) * (1,3)\}, t=0,1.$&{1}/{3}&2 & $(t,4,4+t,12+t)$\\
\hline
6&$\{(27-4t) *(1, 2), 3t *(1, 3), 3* (1,4)\}, $&&&\\
&$t=0,1,2,3.$&{1}/{4}&2 & $(t,5,3+t,10+t)$\\
\hline
7&$\{(27-4t) *(1, 2), 3t *(1, 3), 2* (1,4), $&&&\\
&$(1,r)\},5 \le r \le 12; t=0,1,2,3.$&{1}/{r}&2 & $(t,5,3+t,9+t)$\\
\hline
8&$\{(27-4t) *(1, 2), 3t *(1, 3), (1,4), (1,r_1),$&&&\\
&$(1,r_2)\}, (r_1,r_2) \in I_3; t=0,1,2,3.$&$\frac{1}{r_1}+\frac{1}{r_2}-\frac{1}{4}$&2 & $(t,5,3+t,8+t)$\\
\hline
9&$\{(27-4t)*(1, 2),3t *(1, 3), 3 * (1,5)\}, $&&&\\
&$t=0,1,2,3.$&{1}/{10}&2 & $(t,5,3+t,7+t)$\\
\hline
10&$\{(26-4t) *(1, 2), (3+3t) * (1, 3), (1,4)\}, $&&&\\
&$t=0,1,2,3.$&{1}/{4}&2 & $(0,4,3+t,10+t)$\\
\hline
11&$\{(26-4t) *(1, 2), (3+3t) * (1, 3), (1,r)\},$&&&\\
&$5 \le r \le 12; t=0,1,2,3.$&{1}/{r}&2 & $(0,4,3+t,9+t)$\\
\hline

12&$\{(25-4t) *(1, 2), (5+3t) * (1, 3)\}, t=0,1,2,3.$&{1}/{6}&2 & $(t,3,2+t,8+t)$\\
\hline
13&$\{(26-4t) *(1, 2), (2+3t) * (1, 3), 2 * (1,4)\}, $&&&\\
&$t=0,1,2,3.$&{1}/{6}&2 & $(t,4,2+t,8+t)$\\
\hline
14&$\{(26-4t) *(1, 2), (2+3t) * (1, 3), (1,4), (1,5)\}, $&&&\\
&$t=0,1,2,3.$&{7}/{60}&2 & $(t,4,2+t,7+t)$\\
\hline
15&$\{(26-4t) *(1, 2), (2+3t) * (1, 3), (1,4), $&&&\\
&$(1,6)\}, t=0,1,2,3.$&{1}/{12}&2 & $(t,4,2+t,7+t)$\\
\hline
16&  $\{(25-4t)*(1, 2), (4+3t) * (1, 3), (1,4)\}, $&&&\\
&$t=0,1,2,3.$&{1}/{12}&2 & $(t,3,1+t,6+t)$\\
\hline
17&$\{(26-4t) *(1, 2), (1+3t)*(1, 3), 3 * (1,4)\}, $&&&\\
&$t=0,1,2,3.$&{1}/{12}&2 & $(t,4,1+t,6+t)$\\
\hline
\end{tabular}}\end{center}}
where
\smallskip

{\tiny
$ \begin{array}{ll}I_4 &=\{ (r_1,r_2)| 1/r_1+1/r_2 \ge 1/4, r_i \ge 5\} \\
 &=\{(5,5),\ldots ,(5,20), (6,6), \ldots,(6,12),
(7,7),(7,8),(7,9),(8,8)\} \\
I_3 & = \{(r_1,r_2)| 1/r_1+1/r_2 \ge 1/3, r_i \ge 5\} \\
& = \{(5,5), (5,6), (5,7), (6,6)\}.
\end{array}$
}

\vbox{\centerline{\textbf{Table II-3}}

\begin{center}
\noindent{\tiny
\begin{tabular}{|l|l|l|l|l|}
\hline
 & $B^0(X)$ &  $K_X^3$&$\chi$ &$(P_3,P_4,P_5,P_6)$\\
\hline
1&$\{32 *(1, 2), 5 * (1,3), 2*(1,4), (1,r)\}, 5 \le r.$&$\frac{1}{6}+\frac{1}{r}$&3 & $(0,5,4,13)$\\
\hline
2&$\{(32-4t) *(1, 2), (5+3t) * (1,3), (1,4), $&&&\\
&$(1,r_1), (1,r_2)\}, (r_1,r_2) \in I_6, t \le 1.$&$\frac{1}{r_1}+\frac{1}{r_2}-\frac{1}{12}$&3 & $(t,5,4+t,12+t)$\\
\hline
3&$\{(32-4t) *(1, 2), (5+3t) * (1,3), (1,r_1), $&&&\\
&$(1,r_2), (1,r_3)\}, (r_1,r_2,r_3) \in J, t \le 2.$&$\frac{1}{r_1}+\frac{1}{r_2}+\frac{1}{r_3}-\frac{1}{3}$&3 & $(t,5,4+t,11+t)$\\
\hline

4&$\{(31-4t) *(1, 2), (7+3t) * (1,3), $&&&\\
&$2 * (1,4)\}, t \le 1.$&{1}/{3}&3 & $(t,4,3+t,12+t)$\\
\hline
5&$\{(31-4t) *(1, 2), (7+3t) * (1,3), $&&&\\
&$(1,4), (1,r)\},5 \le r; t \le 2.$&$\frac{1}{12}+\frac{1}{r}$&3 & $(t,4,3+t,11+t)$\\
\hline
6&$\{(31-4t) *(1, 2), (7+3t) * (1,3), $&&&\\
&$(1,r_1), (1,r_2) \}, (r_1,r_2) \in I_4; t \le 3.$&$\frac{1}{r_1}+\frac{1}{r_2}-\frac{1}{6}$&3 & $(t,4,3+t,10+t)$\\
\hline
7&$\{(30-4t) *(1, 2), (10+3t) * (1,3)\}, t=0,1.$&{1}/{3}&3 & $(t,3,3+t,12+t)$\\
\hline

8&$\{(31-4t) *(1, 2), (6+3t) *(1, 3), $&&&\\
&$3* (1,4)\}, t=0,1,2,3.$&{1}/{4}&3 & $(t,4,2+t,10+t)$\\
\hline
9&$\{(31-4t) *(1, 2), (6+3t) *(1, 3),$&&&\\
&$ 2* (1,4), (1,r)\},5 \le r \le 12; t=0,1,2,3.$&{1}/{r}&3 & $(t,4,2+t,9+t)$\\
\hline
10&$\{(31-4t) *(1, 2), (6+3t) *(1, 3), $&&&\\
&$(1,4), (1,r_1),(1,r_2)\}, (r_1,r_2) \in I_3; t \le 3.$&$\frac{1}{r_1}+\frac{1}{r_2}-\frac{1}{4}$&3 & $(t,4,2+t,8+t)$\\
\hline
11&$\{(31-4t)*(1, 2), (6+3t) *(1, 3), $&&&\\
&$3 * (1,5)\}, t=0,1,2,3.$&{1}/{10}&3 & $(t,4,2+t,7+t)$\\
\hline
12&$\{(30-4t) *(1, 2), (9+3t) * (1, 3), $&&&\\
&$(1,4)\}, t=0,1,2,3.$&{1}/{4}&3 & $(0,3,2+t,10+t)$\\
\hline
13&$\{(30-4t) *(1, 2), (9+3t) * (1, 3), $&&&\\
&$(1,r)\},5 \le r \le 12; t=0,1,2,3.$&{1}/{r}&3 & $(0,3,2+t,9+t)$\\
\hline
14&$\{(30-4t) *(1, 2), (8+3t) * (1, 3), $&&&\\
&$2 * (1,4)\}, t=0,1,2,3.$&{1}/{6}&3 & $(t,3,1+t,8+t)$\\
\hline
15&$\{(30-4t) *(1, 2), (8+3t) * (1, 3), $&&&\\
&$(1,4), (1,5)\}, t=0,1,2,3.$&{7}/{60}&3 & $(t,3,1+t,7+t)$\\
\hline
16&$\{(30-4t) *(1, 2), (8+3t) * (1, 3), $&&&\\
&$(1,4), (1,6)\}, t=0,1,2,3.$&{1}/{12}&3 & $(t,3,1+t,7+t)$\\
\hline

17&$\{(30-4t) *(1, 2), (7+3t)*(1, 3), $&&&\\
&$3 * (1,4)\}, t=0,1,2,3.$&{1}/{12}&3 & $(t,3,t,6+t)$\\
\hline

\end{tabular}}\end{center}}
where


{\tiny
$ \begin{array}{ll}I_4 &=\{ (r_1,r_2)| 1/r_1+1/r_2 \ge 1/4, r_i \ge 5\} \\
 &=\{(5,5),\ldots ,(5,20), (6,6), \ldots,(6,12),
(7,7),(7,8),(7,9),(8,8)\} \\
I_3 & = \{(r_1,r_2)| 1/r_1+1/r_2 \ge 1/3, r_i \ge 5\} \\
& = \{(5,5), (5,6), (5,7), (6,6)\}.\\
I_6 & = \{(r_1,r_2)| 1/r_1+1/r_2 \ge 1/6, r_i \ge 5\} \\
& = \{(5,s_5), (6,s_6), (7,s_7), (8,s_8), (9,s_9), (10, s_{10}), (11,11),(11,12), (11,13),(12,12)\},\\
& 5 \le s_1, 6 \le s_2, 7 \le s_7 \le 42, 8 \le s_8 \le 24, 9 \le s_9 \le 18, 10 \le s_{10} \le 15.\\
J & = \{(r_1,r_2,r_3)| 1/r_1+1/r_2+1/r_3 \ge 5/12, r_i \ge 5\} \\
& = \{(5,5,s_1), (5,6,s_2), (5,7,s_3), (5,8,8), (5,8,9), (5,8,10), (5,9,9), (6,6, s_4), (6,7,7),(6,7,8), \\
& (6,7,9),(6,8,8), (7,7,7)\}, 5 \le s_1 \le 60, 6 \le s_2 \le 20, 7 \le s_3 \le 13, 6 \le s_4 \le 12.\\
\end{array}$
}

\end{document}